\documentclass[12pt, twoside]{article}
\usepackage{amsmath,amsthm,amssymb}
\usepackage{times}
\usepackage{enumerate}
\usepackage{extarrows}
\usepackage{color}
\usepackage{tabularx}
\usepackage{multirow}
\DeclareMathOperator*{\esssup}{ess\,sup}

\def\titlerunning#1{\gdef\titrun{#1}}
\makeatletter
\def\author#1{\gdef\autrun{\def\and{\unskip, }#1}\gdef\@author{#1}}
\def\address#1{{\def\and{\\\hspace*{18pt}}\renewcommand{\thefootnote}{}%
\footnote {#1}}%
\markboth{\autrun}{\titrun}}
\makeatother
\def\email#1{e-mail: #1}
\def\subjclass#1{{\renewcommand{\thefootnote}{}%
\footnote{\emph{Mathematics Subject Classification (2010):} #1}}}
\def\keywords#1{\par\medskip
\noindent\textbf{Keywords.} #1}

\newtheorem{theorem}{Theorem}[section]
\newtheorem{lemma}[theorem]{Lemma}
\newtheorem{definition}[theorem]{Definition}
\newtheorem{proposition}[theorem]{Proposition}
\newtheorem{remark}[theorem]{Remark}
\newtheorem{corollary}[theorem]{Corollary}
\newtheorem{example}[theorem]{Example}

\newcounter{Mr}
\newtheorem{Result}[Mr]{\textbf{Main Result}}
\newcounter{Mp}
\newtheorem{Prop}[Mp]{\textbf{Main Proposition}}

\newcommand{\gm}{\gamma}
\newcommand{\R}{\mathbb{R}}
\newcommand{\eps}{\varepsilon}
\newcommand{\Proof}{\begin{proof}}
\newcommand{\End}{\end{proof}}

\numberwithin{equation}{section}

\newcommand{\PreserveBackslash}[1]{\let\temp=\\#1\let\\=\temp}
\newcolumntype{C}[1]{>{\PreserveBackslash\centering}p{#1}}
\newcolumntype{R}[1]{>{\PreserveBackslash\raggedleft}p{#1}}
\newcolumntype{L}[1]{>{\PreserveBackslash\raggedright}p{#1}}

\newcolumntype{I}{!{\vrule width 1pt}}
\newlength\savedwidth

\begin{document}

%%%%% To ease editing, add:

\baselineskip=15pt

%%%%%%%%%%%%%%%%

%% In the running head, give an abbreviation of the title.
\titlerunning{Weak KAM solutions of Hamilton-Jacobi equations}

\title{Weak KAM solutions of Hamilton-Jacobi equations  with  decreasing dependence on unknown functions}

\author{Kaizhi Wang \and Lin Wang \and Jun Yan}

\date{\today}

\maketitle

\address{Kaizhi Wang: School of Mathematical Sciences, Shanghai Jiao Tong University, Shanghai 200240, China; \email{kzwang@sjtu.edu.cn}
\and
Lin Wang: Yau Mathematical Sciences Center, Tsinghua University, Beijing 100084, China; \email{linwang@tsinghua.edu.cn}
\and Jun Yan: School of Mathematical Sciences, Fudan University, Shanghai 200433, China;
\email{yanjun@fudan.edu.cn}}
\subjclass{37J50; 35F21; 35D40}

%%%%%%%%
\begin{abstract}
We consider the Hamilton-Jacobi equation
\[{H}(x,u,Du)=0,\quad x\in M,
\]
where $M$ is a connected, closed and smooth Riemannian manifold,  ${H}(x,u,p)$ satisfies Tonelli conditions with respect to  $p$ and certain decreasing condition with respect to $u$.
Based on a dynamical approach developed   in \cite{WWY,WWY1,WWY2}, we obtain a series of properties for weak KAM solutions (equivalently, viscosity solutions) of the stationary equation and the long time behavior of viscosity solutions of the  evolutionary equation on the Cauchy problem
\begin{equation*}
	\begin{cases}
w_t+{H}(x,w,w_x)=0,\quad (x,t)\in M\times (0,+\infty),\\
w(x,0)=\varphi(x), \quad x\in M.
\end{cases}
\end{equation*}

%% Keywords are optional
\keywords{Hamilton-Jacobi equations, contact Hamiltonian systems, viscosity solutions, Aubry-Mather theory, weak KAM theory}
\end{abstract}

%\newpage
%
%\tableofcontents

%\newpage
%%%%%%%%%%%%%%%%%%%%%%%%%%%%%%%%%%%%%%%%%%%%%%%%%%%%%%%%Sect. 1

\section{Introduction}
\setcounter{equation}{0}
\setcounter{footnote}{0}
\subsection{Main results}
In 1980s, Crandall and Lions introduced the notion of ``viscosity solutions" of scalar nonlinear first order Hamilton-Jacobi equations in \cite{CL}. The theory of viscosity solutions for Hamilton-Jacobi equations with Hamiltonians $H(x,u,p)$,  has been widely studied in the literature (see e.g., \cite{Bardi-Capuzzo-Dolcetta,Ba33,CIL,Ishii,Lions,LPV} and the references therein).  In this paper we aim to understand more about viscosity solutions of both stationary and evolutionary Hamilton-Jacobi equations with Hamiltonians $\bar{H}(x,u,p)$ which satisfy Tonelli conditions with respect to $p$ and certain decreasing condition with respect to $u$ from the weak KAM point of view.

Let $\bar{H}:T^*M\times\mathbb{R}\to \mathbb{R}$, $\bar{H}=\bar{H}(x,u,p)$, be a $C^3$ function satisfying
\begin{itemize}
	\item [\textbf{(H1)}] {\it Strict convexity}:  the Hessian $\frac{\partial^2 \bar{H}}{\partial p^2} (x,u,p)$ is positive definite for all $(x,u,p)\in T^*M\times\R$;
	%\item [\textbf{(H2)}] \textbf{Superlinearity in the Fibers}: For every compact set $I$, $\bar{H}(x,u,p)$ ($u\in I$) is uniformly superlinear growth with respect  to $p$;
	\item [\textbf{(H2)}] {\it Superlinearity}: for every $(x,u)\in M\times\R$, $\bar{H}(x,u,p)$ is  superlinear in $p$;
	\item [\textbf{(H3)}] {\it Moderate decreasing}: there is a constant $\lambda>0$ such that for every $(x,u,p)\in T^{\ast}M\times\R$,
		\begin{equation*}
		-\lambda\leq\frac{\partial \bar{H}}{\partial u}(x,u,p)<0,
		\end{equation*}
\end{itemize}
where $M$ is a connected, closed  and smooth   Riemannian manifold, and $T^*M$ denotes the cotangent bundle of $M$.

Our first main result is devoted to the analysis of the structure of the set of all viscosity solutions of the  stationary Hamilton-Jacobi equation
\begin{align}\label{the00}\tag{HJ$_D$}
\bar{H}(x,u,Du)=0,\quad x\in M.
\end{align}
The letter $D$ represents that $\bar{H}$ is decreasing with respect to the argument $u$. Denote by $\mathcal{VS}(\bar{H})$ the set of all viscosity solutions of $\bar{H}(x,u,Du)=0$.

In order to state our results, let us first recall the notion of admissibility introduced in \cite{WWY2}. For any given $a\in\R$, $\bar{H}(x,a,p)$ is a classical Tonelli Hamiltonian. Ma\~n\'e  critical value  \cite{Mn3}  of $\bar{H}(x,a,p)$ is the unique value of $k$ for which $\bar{H}(x,a,Du)=k$ admits a global viscosity solution. Denote by $c(\bar{H}^a)$ the Ma\~n\'e critical value  of $\bar{H}(x,a,p)$. We say that $\bar{H}(x,u,p)$ is {\em admissible}, if there exists $a\in \R$ such that $c(\bar{H}^a)=0$.
Consider the admissibility assumption:
\begin{itemize}
\item [\textbf{(A)}] {\it Admissibility}:  $\bar{H}(x,u,p)$ is admissible.
\end{itemize}

\begin{Result}[\bf Structure of $\mathcal{VS}(\bar{H})$]\label{thmr1} Assume (H1), (H2) and (H3).
\begin{itemize}
\item [(1)] {\bf Existence}. $\mathcal{VS}(\bar{H})\neq\emptyset$ if and only if condition (A) holds true.

If, in addition, assume (A), then

\item [(2)] {\bf Compactness}. There is a constant $B>0$ such that $\|\bar{v}_-\|_{W^{1,\infty}(M)}\leq B$ for all $\bar{v}_-\in\mathcal{VS}(\bar{H})$.
\item [(3)] {\bf Minimal and maximal elements}. Let $\bar{u}_-:=\min_{\bar{v}_-\in\mathcal{VS}(\bar{H})}\bar{v}_-$. Then $\bar{u}_-\in \mathcal{VS}(\bar{H})$.
The set of all maximal viscosity solutions of equation \eqref{the00} is non-empty.
\item [(4)] {\bf Representation formula}.
Let $\bar{v}_-\in\mathcal{VS}$. Then
\[
\bar{v}_-(x)=\inf_{\xi\in \mathcal{I}_{\bar{v}_-}}\bar{m}(\xi,x),\quad x\in M,
\]
where
\[
\bar{m}(\xi,x):=\liminf_{M\ni y\rightarrow \xi}\inf_{\tau>0}\bar{h}_{y,\bar{v}_-(y)}(x,\tau),\quad x\in M,\quad  \xi\in \mathcal{I}_{\bar{v}_-},
\]
$\mathcal{I}_{\bar{v}_-}$ is a compact subset of $M$ depending only on $\bar{v}_-$, and the function $(x_0,u_0,x,t)\mapsto \bar{h}_{x_0,u_0}(x,t)$ is locally Lipschitz continuous on $M\times \R\times M\times (0,+\infty)$ and it depends only on $\bar{H}$.
\item [(5)] {\bf Regularity}. Let $\bar{v}_-\in \mathcal{VS}$. Then $\bar{v}_-$ is of class $C^{1,1}$ on $\mathcal{I}_{\bar{v}_-}$.
\item [(6)] {\bf Comparison}. Let $\bar{v}'_-$, $\bar{v}''_-\in \mathcal{VS}(\bar{H})$. Then
\begin{itemize}
\item [(i)] if $\bar{v}'_-\leq \bar{v}''_-$ everywhere, then $\emptyset\neq\mathcal{I}_{\bar{v}''_-}\subset \mathcal{I}_{\bar{v}'_-}\subset \mathcal{I}_{\bar{u}_-}$,
where $\bar{u}_-:=\min_{\bar{v}_-\in\mathcal{VS}(\bar{H})}\bar{v}_-$.
\item [(ii)] if there is a neighborhood of $\mathcal{I}_{\bar{v}'_-}$, denoted by $\mathcal{O}$, such that
\[
\bar{v}''_-|_{\mathcal{O}}\leq \bar{v}'_-|_{\mathcal{O}},
\]
then $\bar{v}''_-\leq \bar{v}'_-$ everywhere.
\item [(iii)] if
\[
\mathcal{I}_{\bar{v}''_-}=\mathcal{I}_{\bar{v}'_-},\quad \bar{v}''_-|_{\mathcal{O}}= \bar{v}'_-|_{\mathcal{O}},
\]
then $\bar{v}''_-= \bar{v}'_-$ everywhere.
\end{itemize}
\end{itemize}
\end{Result}

\begin{remark}
We make some explanations for  notions and notations in the above result.
\begin{itemize}
\item [(1)] See definitions of $\mathcal{I}_{\bar{v}_-}$ and  $\bar{h}_{\cdot,\cdot}(\cdot,\cdot)$ in Subsection 2.2 below.
\item [(2)] Note that $(C(M,\mathbb{R}),\geq)$ is a partially ordered set, where $v'$, $v''\in C(M,\mathbb{R})$, $v'\geq v''$ if and only if $v'(x)\geq v''(x)$ for all $x\in M$.
$\bar{v}_-^*\in \mathcal{VS}(\bar{H})$ is called a maximal element of $(\mathcal{VS}(\bar{H}),\geq)$ if for all $\bar{v}'_-\in\mathcal{VS}(\bar{H})$, $\bar{v}'_-\geq \bar{v}_-^*$ implies $\bar{v}'_-=\bar{v}_-^*$ everywhere.	
\end{itemize}
\end{remark}

The second aim of this paper to study the long time behavior of the viscosity solution of the Cauchy problem
\begin{equation}\label{evll}\tag{HJ$_C$}
\begin{cases}
w_t+\bar{H}(x,w,w_x)=0,\quad (x,t)\in M\times(0,\infty),\\
w(x,0)=\varphi(x), \quad x\in M,
\end{cases}
\end{equation}
where $\varphi\in C(M,\mathbb{R})$ is the initial data.

\begin{Result}[\bf Long time behavior of viscosity solutions of \eqref{evll}]\label{thmr2}
Assume (H1), (H2), (H3) and (A). Given $\varphi\in C(M,\R)$, let ${w}(x,t)$ denote the unique viscosity solution of \eqref{evll}. Denote by $\bar{u}_+$ the unique forward weak KAM solution of equation \eqref{the00}.
\begin{itemize}
\item [(1)] {\bf Uniform boundedness}.
    \begin{itemize}
    \item  [(1-1)]$w(x,t)$ is bounded on $M\times[0,+\infty)$ if and only if $\varphi$ satisfies:
    \begin{itemize}
    \item [(a)]  $\varphi\geq \bar{u}_+$ everywhere.
    \item [(b)] there exists $x_0\in M$ such that $\varphi(x_0)=\bar{u}_+(x_0)$.
    \end{itemize}
\item  [(1-2)]If one of conditions (a) and (b) is not satisfied, then we have two cases:
\begin{itemize}
\item [(i)] if there is $x_0\in M$ such that $\varphi(x_0)<\bar{u}_+(x_0)$, then $\lim_{t\to+\infty}w(x,t)=-\infty$ uniformly on $x\in M$.
\item [(i)] if $\varphi>\bar{u}_+$ everywhere, then
$\lim_{t\to+\infty}w(x,t)=+\infty$ uniformly on $x\in M$.
\end{itemize}
\item [(1-3)] There is $K>0$ such that for each $\varphi$ satisfying (a) and (b), there is $T_{\varphi}>0$ such that
    \[
    |w(x,t)|\leq K\quad  (x,t)\in M\times(T_{\varphi},+\infty).
    \]
\end{itemize}

%\item [(ii)] for  $\varphi$ satisfying (a) and (b),  define
%\[\check{w}(x):=\liminf_{t\rightarrow+\infty}w(x,t)\quad and \quad \hat{w}(x):=\limsup_{t\rightarrow+\infty}w(x,t).\]
%Then $\check{w}$, $\hat{w}\in\mathcal{VS}(\bar{H})$. In particular, if $\mathcal{VS}(\bar{H})$ is a singleton, then the unique viscosity solution of equation \eqref{the00} is $w_\infty(x):=\lim_{t\rightarrow+\infty}w(x,t)$, $x\in M$.
%\item [(iv)] if $\varphi$ is a viscosity subsolution of  (\ref{the00}) and it satisfies (a) and (b), then the uniform limit \[w_\infty(x):=\lim_{t\rightarrow+\infty}w(x,t)\] exists and it is a viscosity solution of (\ref{the00}).

\item [(2)] {\bf Sufficient conditions for convergence}.
\begin{itemize}
 \item [(2-1)] For any $\varphi\in C(M,\R)$, if $\bar{u}_+\leq \varphi\leq \bar{u}_-$, then $\lim_{t\to +\infty}\bar{T}^-_t\varphi=\bar{u}_-$ uniformly on $x\in M$.
 \item [(2-2)] If $\mathcal{VS}(\bar{H})=\{\varphi_\infty\}$ is a singleton, then $\varphi_\infty=\lim_{t\rightarrow+\infty}\bar{T}_t^-\varphi$ for all $\varphi$ satisfying (a) and (b).
\item [(2-3)]
Let $\bar{v}^*_-$ be a maximal element of $\mathcal{VS}(\bar{H})$. For  $\varphi$ satisfying (a), (b) and
 $\varphi\geq \bar{v}^*_-$, $w(x,t)$ converges uniformly to $\bar{v}^*_-$ as $t$ tends to infinity.
  \end{itemize}
\end{itemize}
\end{Result}

%\begin{remark}
%We remark here that our method presents a more dynamical flavor and we do not pursue optimality of assumptions on $\bar{H}$, namely (H1)-(H3) and $C^3$, under which Main Result 1 and Main Result 2 hold true. As pointed out by the referee, certain parts of the results requires less regularity of $\bar{H}$.
%\end{remark}

\begin{corollary}\label{ad-c}
Let $\bar{H}=\bar{H}(x,u,p):T^*\mathbb{S}\times\R\to\mathbb{R}$ be a $C^3$ Hamiltonian satisfying (H1), (H2) and (H3) and (A), where $\mathbb{S}$ is the unit circle. Given $\varphi\in C(M,\R)$, let ${w}(x,t)$ denote the unique viscosity solution of \eqref{evll}.
   If the contact vector field generated by $\bar{H}(x,u,p)$ has no singular points, then $\mathcal{VS}(\bar{H})=\{\varphi_\infty\}$ is a singleton and $w(x,t)$ converges uniformly to $\varphi_\infty$ as $t\to+\infty$ for all $\varphi$ satisfying (a) and (b) in Item (1-1) of Main Result \ref{thmr2}.
\end{corollary}

\begin{example}
Let $\bar{H}(x,u,p)=-u+\frac{1}{2}(p+p_0)^2+\cos x$, $x\in\mathbb{S}$, $p_0\notin [-1,1]$.
It is clear that $\bar{H}$ satisfies (H1), (H2), (H3) and (A). Moreover, the contact vector field associated with $\bar{H}(x,u,p)$ has no singular points. By  Corollary \ref{ad-c}, we deduce that the following equation:
\begin{equation}\label{hj777}
	-u+\frac{1}{2}(Du+p_0)^2+\cos x=0
\end{equation}
admits a unique viscosity solution, and thus the convergence result holds true.
\end{example}

The discounted Hamilton-Jacobi equation is a specific example of Hamilton-Jacobi equations whose Hamiltonian $H(x,u,p)$ is strictly increasing in the argument $u$. The vanishing discount problem for discounted Hamilton-Jacobi equations has been widely studied by using different approaches, see  \cite{Dav,Dav1,Go,Go1,i1,i2,It,MT,S}. Recently, some new processes have been made on the vanishing contact structure problem which is a generalization of the vanishing discount problem, see \cite{CCIZ,ZC}.
We refer the readers to \cite{MS} (see also \cite{MK}) for the Aubry-Mather theory for conformally symplectic systems which are closely related to discounted Hamilton-Jacobi equations.

\subsection{Strategy of the proofs of main results}

We will sketch the strategy of the proofs of main results in this part.
%In \cite{WWY1}, we  introduced forward and backward solution semigroups, denoted by $\{\tilde{T}^-_t\}_{t\geq 0}$ and $\{\tilde{T}^+_t\}_{t\geq 0}$, associated with Hamiltonians $\tilde{H}$ satisfying (H1), (H2) and the uniform Lipschitz condition in $u$:
%\begin{itemize}
%	\item [\textbf{(L)}] {\it Lipschitz continuity}: there is a constant $\lambda>0$ such that for every $(x,u,p)\in T^{\ast}M\times\R$,
%		\begin{equation*}
%		\Big|\frac{\partial \tilde{H}}{\partial u}(x,u,p)\Big|\leq \lambda.
%		\end{equation*}
%\end{itemize}
%
%\begin{remark}
%For Hamiltonians $\tilde{H}(x,u,p)$ satisfying (H1), (H2) and (L), backward weak KAM solutions and viscosity solutions of $\tilde{H}(x,u,Du)=0$ are the same.
%\end{remark}
%
Let
 \[
 H(x,u,p):=\bar{H}(x,-u,-p),\quad \forall (x,u,p)\in T^*M\times\R.
 \]
Since $\bar{H}$ satisfies (H1), (H2) and (H3), then $H$ satisfies (H1), (H2) and the following condition.
\begin{itemize}
	\item [\textbf{(H3')}] {\it Moderate increasing}: there is a constant $\lambda>0$ such that for every $(x,u,p)\in T^{\ast}M\times\R$,
		\begin{equation*}
		0<\frac{\partial H}{\partial u}(x,u,p)\leq \lambda,
		\end{equation*}
\end{itemize}
By definition,  $\bar{H}$ is admissible if and only if $H$ is admissible. By \cite[Appendix B]{WWY2}, the admissibility of ${H}$ is equivalent to
\begin{itemize}
\item [\textbf{(S)}] {\it Solvability}:  $H(x,u,Du)=0$ has a viscosity solution.
\end{itemize}

Denote by $\{T^-_t\}_{t\geq 0}$ (resp. $\{\bar{T}^-_t\}_{t\geq 0}$) the backward solution semigroup associated with $H$ (resp. $\bar{H}$), by $\{T^+_t\}_{t\geq 0}$ (resp. $\{\bar{T}^+_t\}_{t\geq 0}$) the forward solution semigroup associated with $H$ (resp. $\bar{H}$), see Subsection 2.2 below for the definitions of $T^{\pm}_t$.

The following result \cite[Proposition 2.8]{WWY2} is the starting point of  the main results in this paper.
\begin{proposition}\label{i}
Let $\varphi\in C(M,\R)$. Then
\begin{align}\label{sg}
-T^+_t(-\varphi)=\bar{T}^-_t\varphi,\quad -T^-_t(-\varphi)=\bar{T}^+_t\varphi,\quad \forall t\geq 0.
\end{align}
\end{proposition}
Consider
\begin{align}\label{hj}\tag{HJ$_I$}
H(x,u,Du)=0,\quad x\in M.
\end{align}
The letter $I$ represents that $H$ is increasing with respect to the argument $u$. We use $\mathcal{S}_-$ (resp. $\mathcal{S}_+$) to denote the set of all backward (resp. forward) weak KAM solutions of equation \eqref{hj} (see Definition \ref{bwkam} below). Let $u\in \mathcal{S}_+$. Then $u$
is exactly a viscosity solution of
\[-H(x,u,Du)=0\quad \text{in}\ \ M,\]
and, if $v:=-u$, then $v$ is a viscosity solution of
\[H(x,-u,-Du)=0\quad \text{in}\ \ M.\]

Let   $\mathcal{VS}(H)$ be the set of all viscosity solutions of $H(x,u,Du)=0$. Let $\bar{H}(x,t,p):=H(x,-t,-p)$. Then we have
\[\mathcal{S}_+=\mathcal{VS}(-H),\ -\mathcal{S}_+=\mathcal{VS}(\bar{H}),\ \mathcal{S}_-=\mathcal{VS}(H),\ -\mathcal{S}_-=\mathcal{VS}(-\bar{H}).\]

For evolutionary equations, it is well known that $u:=T_t^+\varphi$ is a unique viscosity solution to the Cauchy problem
\begin{equation*}
	\begin{cases}
u_t-{H}(x,u,u_x)=0,\quad (x,t)\in M\times (0,+\infty),\\
u(x,0)=\varphi(x), \quad x\in M,
\end{cases}
\end{equation*}
and that $v:=-T_t^+\varphi$ is a unique viscosity solution to the Cauchy problem
\begin{equation*}
	\begin{cases}
v_t+{H}(x,-v,-v_x)=0,\quad (x,t)\in M\times (0,+\infty),\\
v(x,0)=-\varphi(x), \quad x\in M.
\end{cases}
\end{equation*}

{\bf Consequently, in order to study viscosity solutions of equation \eqref{the00}}, {\bf we only need to study forward weak KAM solutions of equation \eqref{hj}}.

\begin{remark}\label{rem11}
 Under assumptions (H1), (H2), (H3') and (S), $\mathcal{S}_-$ is a singleton.
It is clear that if $\mathcal{S}_+$ is non-empty, then it may not be a singleton. Consider the following example from \cite{WWY2},
\begin{equation}\label{conterex}
{u}+\frac{1}{2}|D{u}|^2=0,\quad x\in\mathbb{S},
\end{equation}
where $\mathbb{S}:=(-\frac 12,\frac 12]$ denotes the unit circle.
Let $u_1$ be the even 1-periodic extension of $u(x)=-\frac 12 x^2$ in $[0,\frac 12]$.
Then both $u_1$ and $u_2\equiv 0$ are forward weak KAM solutions of equation \eqref{conterex}.
\end{remark}

\subsubsection{Outline of the proof of Main Result \ref{thmr1}}

In view of Remark \ref{rem11}, in order to prove Main Result \ref{thmr1}, it suffices to show Main Propositions \ref{ii}, \ref{iii}, \ref{vi},   \ref{345}  below.

\medskip

By Proposition \ref{pairxx}, under assumptions (H1), (H2) and (H3'), condition (S) is a sufficient condition for the non-emptiness of $\mathcal{S}_+$. Here, we attempt to show that condition (S) is also a necessary condition.
\begin{Prop}[\bf Existence]\label{ii}
The set $\mathcal{S}_+$ is non-empty if and only if condition (S) holds true.
\end{Prop}
By \cite[Theorem B.1]{WWY2} condition (S) holds if and only if the set $\mathcal{S}_-$ is non-empty.
So, we have

\medskip
\centerline{$\mathcal{S}_-\neq\emptyset$ $\iff$$\mathcal{S}_+\neq\emptyset$ $\iff$ condition (S) holds true.}

\bigskip

{\em From now on to the end of this section, we assume that condition (S) holds true}.

\bigskip
By \cite[Proposition A.1, Theorem B.1]{WWY2}, $\mathcal{S}_-$ is a singleton.
Denote by $u_-$ the unique backward weak KAM solution of equation \eqref{hj}, by $v_+$ an arbitrary forward weak KAM solution of equation \eqref{hj}.
Recall that each $v_+$ is Lipschitz continuous on $M$ \cite[Lemma 4.1]{WWY2}. Thus, $v_+\in W^{1,\infty}(M)$. Define
\[
\|v_+\|_{W^{1,\infty}(M)}:=\esssup_M\left(|v_+|+|Dv_+|\right).
\]

\begin{Prop}[\bf Compactness]\label{iii}
There is a constant $B>0$ such that $\|v_+\|_{W^{1,\infty}(M)}\leq B$ for all $v_+\in\mathcal{S}_+$.
\end{Prop}
More precisely, we show that there exist  constants $K>0$ and $\kappa>0$ such that  for each $v_+\in \mathcal{S}_+$, $\|v_+\|_\infty\leq K$ and $v_+$ is $\kappa$-Lipschitz continuous on $M$.

It is clear that $(C(M,\mathbb{R}),\leq)$ is a partially ordered set, where $v'_+\leq v''_+$ if and only if $v'_+(x)\leq v''_+(x)$ for all $x\in M$.

\begin{definition}
$v_+\in \mathcal{S}_+$ is called a minimal forward weak KAM solution of equation \eqref{hj}, if $v_+$ is a minimal element of $(\mathcal{S}_+,\leq)$, i.e., for all $v'_+\in\mathcal{S}_+$, $v'_+\leq v_+$ implies $v'_+=v_+$ everywhere.	
\end{definition}

\begin{Prop}[\bf Maximal and minimal elements]\label{vi}
Let $u_+:=\max_{v_+\in\mathcal{S}_+}v_+$. Then $u_+\in \mathcal{S}_+$.
The set of all minimal forward weak KAM solutions of equation \eqref{hj} is non-empty.
\end{Prop}

Let $\mathcal{I}_{v_+}:=\{x\in M\ |\ v_+(x)=u_-(x)\}$ for each $v_+\in \mathcal{S}_+$. We showed in \cite{WWY2} that $\mathcal{I}_{v_+}$ is a non-empty compact subset of $M$. We provide a representation formula for forward weak KAM solutions of equation \eqref{hj} as follows.
\begin{Prop}[\bf Representation formula]\label{four2}
	Let $v_+\in \mathcal{S}_+$. Define
	\[
	m(x,\xi):=\limsup_{y\to \xi}\sup_{\tau>0}h^{y,v_+(y)}(x,\tau),\quad x\in M, \quad\xi\in \mathcal{I}_{v_+}.
	\]
	Then
	\[
	v_+(x)=\sup_{\xi\in \mathcal{I}_{v_+}}m(x,\xi),\quad x\in M.
	\]
\end{Prop}

Any forward weak KAM solution $v_+$ is more regular than Lipschitz on $\mathcal{I}_{v_+}$.
\begin{Prop}[\bf Regularity]\label{3}
	Let $v_+\in \mathcal{S}_+$. Then $v_+$ is of class $C^{1,1}$ on $\mathcal{I}_{v_+}$.
\end{Prop}

\begin{Prop}[\bf Comparision]\label{345} Let $v'_+$, $v''_+\in \mathcal{S}_+$. Then
\begin{itemize}
\item [(1)] if $v'_+\leq v''_+$ everywhere, then $\emptyset\neq\mathcal{I}_{v'_+}\subset \mathcal{I}_{v''_+}\subset \mathcal{I}_{u_+}$, where $u_+:=\max_{v_+\in\mathcal{S}_+}v_+$.
\item [(2)] if there is a neighborhood of $\mathcal{I}_{v'_+}$, denoted by $\mathcal{O}$, such that
\[
v''_+|_\mathcal{O}\geq v'_+|_\mathcal{O},
\]
then $v_+''\geq v'_+$ everywhere.
\item [(3)] if
\[
\mathcal{I}_{v''_+}=\mathcal{I}_{v'_+},\quad v''_+|_\mathcal{O}= v'_+|_\mathcal{O},
\]
then $v_+''= v'_+$ everywhere.
\end{itemize}
\end{Prop}

\subsubsection{Outline of the proof of Main Result \ref{thmr2}}

By Proposition \ref{i}, Main Result \ref{thmr2} is a direct consequence of the following result which is concerned with the  long  time behavior of $T^+_t\varphi$ for a class of $\varphi\in C(M,\R)$.

\begin{Prop}[\bf Uniform boundedness]\label{v}
	Given $\varphi\in C(M,\R)$, there hold
\begin{itemize}
\item [(1)] the family $\{T^+_t\varphi\}_{t\geq 0}$ is uniformly bounded on $M$ if and only if $\varphi$ satisfies:
\begin{itemize}
\item [(a')]  $\varphi\leq u_-$ everywhere.
\item [(b')] there exists $x_0\in M$ such that $\varphi(x_0)=u_-(x_0)$.
\end{itemize}
\item [(2)] if one of conditions (a') and (b') is not satisfied, then we have two cases:
\begin{itemize}
\item [(i)] if there is $x_0\in M$ such that $\varphi(x_0)>u_-(x_0)$, then $\lim_{t\to+\infty}T^+_t\varphi(x)=+\infty$ uniformly on $x\in M$.
\item [(ii)] if $\varphi<u_-$ everywhere, then $\lim_{t\to+\infty}T^+_t\varphi(x)=-\infty$.
\end{itemize}
\item [(3)] there is $K>0$ such that for each
$\varphi$ satisfying (a') and (b'), there is $T_{\varphi}>0$ such that

\[
|T^+_t\varphi(x)|\leq K,\quad \forall (x,t)\in M\times [0,+\infty).
\]
\end{itemize}
\end{Prop}

\begin{Prop}[\bf Sufficient conditions for convergence]\label{77iii}
Given $\varphi\in C(M,\R)$, there hold
\begin{itemize}
\item [(1)] for any $\varphi\in C(M,\R)$, if $u_+\leq \varphi\leq u_-$, then $\lim_{t\to +\infty}T^+_t\varphi=u_+$ uniformly on $x\in M$.	
\item [(2)] if $\mathcal{S}_+=\{\varphi_\infty\}$ is a singleton, then $\varphi_\infty=\lim_{t\rightarrow+\infty}T_t^+\varphi$ for all $\varphi$ satisfying (a') and (b').
\item [(3)] for  $\varphi$ satisfying (a') and (b'), if there exists a minimal forward weak KAM solution $v^*_+$ of equation \eqref{hj} such that $\varphi\leq v^*_+$, then $\{T_t^+\varphi\}_{t\geq 0}$ converges uniformly to $v^*_+$ as $t$ tends to infinity.
     \end{itemize}
\end{Prop}

%
%
%\subsection{Notations}
%We give a list of symbols in Table 1 used to denote forward and backward weak KAM solutions of both equations \eqref{the00} with Hamiltonian $\bar{H}$ and \eqref{hj} with Hamiltonian ${H}$.
%
%\begin{table}[htbp]\label{tabe1}
%\centering
%\begin{tabular}{c|c|c}
%\whline  & backward weak KAM solution & forward weak KAM solution\\
%\whline \multirow{2}{*}{\eqref{hj}} & \multirow{2}{*}{$\mathcal{S}_-=\{u_-\}$ }& \multirow{2}{*}{$v_+\in \mathcal{S}_+$,\ \ \ $u_+:=\max_{\mathcal{S}_+} v_+\in\mathcal{S}_+$}\\
%&&\\
%\hline \multirow{2}{*}{\eqref{the00}} & \multirow{2}{*}{$\bar{v}_-\in \mathcal{VS}$,\ \ \ $\bar{u}_-:=\min_{\mathcal{VS}} \bar{v}_-\in\mathcal{VS}$} & \multirow{2}{*}{$\bar{\mathcal{S}}_+=\{\bar{u}_+\}$ }\\
%&&\\
%\hline \multirow{2}{*}{Relation} & \multicolumn{2}{c}{ \multirow{2}{*}{${H}(x,-u,-p)=\bar{H}(x,u,p)$,\ \ \  $\bar{v}_-=-v_+$,\ \ \ $\bar{u}_+=-u_-$,\ \ \ $\bar{u}_-=-u_+$}}\\
%\\
%\whline \multirow{2}{*}{$\mathcal{I}_{\bar{v}_-}=\mathcal{I}_{v_+}$} & \multicolumn{2}{c}{ \multirow{2}{*}{$\mathcal{I}_{\bar{v}_-}=\mathcal{I}_{v_+}=\{x\in M: v_+(x)=u_-(x)\}=\{x\in M: \bar{v}_-(x)=\bar{u}_+(x)\}$}}\\
%\\
%\whline
%\end{tabular}
%\caption{Notations corresponding to \eqref{hj} and \eqref{the00}}
%\end{table}
%

The rest of the paper is organized as follows. In Section 2 we recall preliminaries on contact Hamiltonian systems. We give the proofs of Main Result \ref{thmr1} in Section 3, Main Result \ref{thmr2} and Corollary \ref{ad-c}  in Section 4. Appendix A devotes to the relation between weak KAM solutions and solution semigroups.

%%%%%%%%%%%%%%%%%%%%%%%%%%%%%%%%%%%%%%%

%==================================================================================================Sect. 2

\section{Preliminaries}
\subsection{Settings}
 We choose, once and for all, a $C^\infty$ Riemannian metric $g$ on $M$. There is a canonical way to associate to it a Riemannian metric on
$TM$ and $T^*M$, respectively. Denote by $d(\cdot,\cdot)$ the distance
function defined by $g$ on $M$. We use the same symbol $\|\cdot\|_x$ to denote the
norms induced by the Riemannian metrics on $T_xM$ and $T^*_xM$ for $x\in M$, and by
$\langle \cdot,\cdot\rangle_x$ the canonical pairing between the tangent space $T_xM$ and the cotangent space $T^*_xM$. $C(M,\mathbb{R})$ stands for  the space of continuous functions on $M$, $\|\cdot\|_\infty$ denotes the supremum norm on it.

In a series of papers \cite{WWY,WWY1,WWY2}, the authors developed a new variational and dynamical approach for contact Hamilton equations

\begin{align}\label{c}\tag{CH}
\left\{
        \begin{array}{l}
        \dot{x}=\frac{\partial H}{\partial p}(x,u,p),\\
        \dot{p}=-\frac{\partial H}{\partial x}(x,u,p)-\frac{\partial H}{\partial u}(x,u,p)p,\qquad (x,u,p)\in T^*M\times\mathbb{R},\\
        \dot{u}=\frac{\partial H}{\partial p}(x,u,p)\cdot p-H(x,u,p).
         \end{array}
         \right.
\end{align}
We refer the readers to \cite{CCWY} for another variational formulation of (\ref{c}),
which was obtained from Lagrangian formalism using an alternative method.
See \cite{CCJWY,WY} for a time-dependent version of the variational principle for contact Hamiltonian systems.

%We use $\mathcal{L}: T^*M\rightarrow TM$ to denote  the Legendre transformation. Let
%$\bar{\mathcal{L}}:=(\mathcal{L}, Id)$, where $Id$ denotes the identity map from $\R$ to $\R$. Then $\bar{\mathcal{L}}$ denote a diffeomorphism from $T^*M\times\R$ to $TM\times\R$. By $\bar{\mathcal{L}}$,
The contact Lagrangian $L(x,u, \dot{x})$ associated to $H(x,u,p)$ is defined by
\[
L(x,u, \dot{x}):=\sup_{p\in T^*_xM}\left\{\langle \dot{x},p\rangle-H(x,u,p)\right\}.
\]
Since $H$ satisfies (H1), (H2) and (H3'), we have:
\begin{itemize}
\item [\textbf{(L1)}] {\it Strict convexity}:  the Hessian  $\frac{\partial^2 L}{\partial {\dot{x}}^2} (x,u,\dot{x})$ is positive definite for all $(x,u,\dot{x})\in TM\times\R$;
%\item [\textbf{(L2)}] \textbf{Superlinearity in the Fibers}: For every compact set $I$, $L(x,u,\dot{x})$ ($u\in I$) is uniformly superlinear growth with respect  to $\dot{x}$;
    \item [\textbf{(L2)}] {\it Superlinearity}: for every $(x,u)\in M\times\R$, $L(x,u,\dot{x})$ is superlinear in $\dot{x}$;
\item [\textbf{(L3')}] {\it Moderate decreasing}: there is a constant $\lambda>0$ such that for every $(x,u,\dot{x})\in TM\times\R$,
                        \begin{equation*}
                        -\lambda\leq \frac{\partial L}{\partial u}(x,u,\dot{x})<0.
                        \end{equation*}
\end{itemize}

\subsection{Notions and Notations}

Recall implicit variational principles introduced in \cite{WWY} for contact Hamilton's equations (\ref{c}).
	For any given $x_0\in M$, $u_0\in\mathbb{R}$, there exist two continuous functions $h_{x_0,u_0}(x,t)$ and $h^{x_0,u_0}(x,t)$ defined on $M\times(0,+\infty)$ satisfying	
\begin{equation}\label{baacf1}
h_{x_0,u_0}(x,t)=u_0+\inf_{\substack{\gamma(0)=x_0 \\  \gamma(t)=x} }\int_0^tL\big(\gamma(\tau),h_{x_0,u_0}(\gamma(\tau),\tau),\dot{\gamma}(\tau)\big)d\tau,
\end{equation}
\begin{align}\label{2-3}
h^{x_0,u_0}(x,t)=u_0-\inf_{\substack{\gamma(t)=x_0 \\  \gamma(0)=x } }\int_0^tL\big(\gamma(\tau),h^{x_0,u_0}(\gamma(\tau),t-\tau),\dot{\gamma}(\tau)\big)d\tau,
\end{align}
where the infimums are taken among the Lipschitz continuous curves $\gamma:[0,t]\rightarrow M$.
Moreover, the infimums in (\ref{baacf1}) and \eqref{2-3} can be achieved.
If $\gamma_1$ and $\gamma_2$ are curves achieving the infimums \eqref{baacf1} and \eqref{2-3} respectively, then $\gamma_1$ and $\gamma_2$ are of class $C^1$.
Let
\begin{align*}
x_1(s)&:=\gamma_1(s),\quad u_1(s):=h_{x_0,u_0}(\gamma_1(s),s),\,\,\,\qquad  p_1(s):=\frac{\partial L}{\partial \dot{x}}(\gamma_1(s),u_1(s),\dot{\gamma}_1(s)),\\
x_2(s)&:=\gamma_2(s),\quad u_2(s):=h^{x_0,u_0}(\gamma_1(s),t-s),\quad   p_2(s):=\frac{\partial L}{\partial \dot{x}}(\gamma_2(s),u_2(s),\dot{\gamma}_2(s)).
\end{align*}
Then $(x_1(s),u_1(s),p_1(s))$ and $(x_2(s),u_2(s),p_2(s))$ satisfy equations \eqref{c} with
\begin{align*}
x_1(0)=x_0, \quad x_1(t)=x, \quad \lim_{s\rightarrow 0^+}u_1(s)=u_0,\\
x_2(0)=x, \quad x_2(t)=x_0, \quad \lim_{s\rightarrow t^-}u_2(s)=u_0.
\end{align*}

We call $h_{x_0,u_0}(x,t)$ (resp. $h^{x_0,u_0}(x,t)$) a {\it forward (resp. backward) implicit action function} associated with $L$
and the curves achieving the infimums in (\ref{baacf1}) (resp. \eqref{2-3}) minimizers of $h_{x_0,u_0}(x,t)$ (resp. $h^{x_0,u_0}(x,t)$).
The relation between forward and backward implicit action functions is as follows: for any given $x_0$, $x\in M$, $u_0$, $u\in\mathbb{R}$ and $t>0$,  $h_{x_0,u_0}(x,t)=u$ if and only if $h^{x,u}(x_0,t)=u_0$.

Let us recall two  semigroups of operators introduced in \cite{WWY1}.  Define a family of nonlinear operators $\{T^-_t\}_{t\geq 0}$ from $C(M,\mathbb{R})$ to itself as follows. For each $\varphi\in C(M,\mathbb{R})$, denote by $(x,t)\mapsto T^-_t\varphi(x)$ the unique continuous function on $ (x,t)\in M\times[0,+\infty)$ such that
\[
T^-_t\varphi(x)=\inf_{\gamma}\left\{\varphi(\gamma(0))+\int_0^tL(\gamma(\tau),T^-_\tau\varphi(\gamma(\tau)),\dot{\gamma}(\tau))d\tau\right\},
\]
where the infimum is taken among the absolutely continuous curves $\gamma:[0,t]\to M$ with $\gamma(t)=x$.  Let $\gamma$ be a curve achieving the infimum, and $x(s):=\gamma(s)$, $u(s):=T_t^-\varphi(x(s))$, $p(s):=\frac{\partial L}{\partial \dot{x}}(x(s),u(s),\dot{x}(s))$.
Then $(x(s),u(s),p(s))$ satisfies equations (\ref{c}) with $x(t)=x$.

In \cite{WWY1} we proved that $\{T^-_t\}_{t\geq 0}$ is a semigroup of operators and the function $(x,t)\mapsto T^-_t\varphi(x)$ is a viscosity solution of $w_t+H(x,w,w_x)=0$ with $w(x,0)=\varphi(x)$. Thus, we call $\{T^-_t\}_{t\geq 0}$ the \emph{backward solution semigroup}.

Similarly, one can define another semigroup of operators $\{T^+_t\}_{t\geq 0}$, called the \emph{forward solution semigroup}, by

\begin{equation*}\label{fixufor}
T^+_t\varphi(x)=\sup_{\gamma}\left\{\varphi(\gamma(t))-\int_0^tL(\gamma(\tau),T^+_{t-\tau}\varphi(\gamma(\tau)),\dot{\gamma}(\tau))d\tau\right\},
\end{equation*}
where the supremum is taken among the absolutely continuous curves $\gamma:[0,t]\to M$ with $\gamma(0)=x$. Let $\gamma$ be a curve achieving the supremum, and $x(s):=\gamma(s)$, $u(s):=T_{t-s}^+\varphi(x(s))$, $p(s):=\frac{\partial L}{\partial \dot{x}}(x(s),u(s),\dot{x}(s))$.
Then $(x(s),u(s),p(s))$ satisfies equations (\ref{c}) with $x(0)=x$.

Following Fathi \cite{Fat-b}, one can define weak KAM solutions of  equation \eqref{hj}.
\begin{definition}\label{bwkam}
	A function $u\in C(M,\mathbb{R})$ is called a backward weak KAM solution of \eqref{hj} if
	\begin{itemize}
		\item [(i)] for each continuous piecewise $C^1$ curve $\gamma:[t_1,t_2]\rightarrow M$, we have
		\[
		u(\gamma(t_2))-u(\gamma(t_1))\leq\int_{t_1}^{t_2}L(\gamma(s),u(\gamma(s)),\dot{\gamma}(s))ds;
		\]
		\item [(ii)] for each $x\in M$, there exists a $C^1$ curve $\gamma:(-\infty,0]\rightarrow M$ with $\gamma(0)=x$ such that
		\begin{align}\label{cali1}
		u(x)-u(\gamma(t))=\int^{0}_{t}L(\gamma(s),u(\gamma(s)),\dot{\gamma}(s))ds, \quad \forall t<0.
		\end{align}
	\end{itemize}
Similarly, 	a function $u\in C(M,\mathbb{R})$ is called a forward weak KAM solution of of \eqref{hj} if it satisfies (i) and
for each $x\in M$, there exists a $C^1$ curve $\gamma:[0,+\infty)\rightarrow M$ with $\gamma(0)=x$ such that
		\begin{align}\label{cali2}
		u(\gm(t))-u(x)=\int_{0}^{t}L(\gamma(s),u(\gamma(s)),\dot{\gamma}(s))ds,\quad \forall t>0.
		\end{align}
\end{definition}
	We denote by $\mathcal{S}_-$ (resp. $\mathcal{S}_+$) the set of backward (resp. forward) weak KAM solutions of equation \eqref{hj}. By the analogy of \cite{Fat-b}, (i) of Definition \ref{bwkam} reads that $u$  is dominated by $L$, denoted by $u\prec L$. The curves in \eqref{cali1} and \eqref{cali2} are called {\it  $(u,L,0)$-calibrated curves}.

\subsection{Some useful propositions}
Proposition \ref{hupin}, Proposition\ref{repaction} and Proposition \ref{pr4.5} below hold if $H$ satisfies (H1),(H2) and $|\frac{\partial H}{\partial u}|\leq \lambda$.
\begin{proposition}\label{hupin}
For each $(x_0,u_0,x,t)\in M\times\R\times M\times (0,+\infty)$,
\begin{align}
\bar{h}_{x_0,u_0}(x,t)=-h^{x_0,-u_0}(x,t),\quad \bar{h}^{x_0,u_0}(x,t)=-h_{x_0,-u_0}(x,t).
\end{align}
\end{proposition}

\begin{proof}
We only prove the first equality, since the second one is similar. By contradiction, we assume there exists $(y_0,t_0)\in M\times (0,+\infty)$ such that
\[\bar{h}_{x_0,u_0}(y_0,t_0)<-h^{x_0,-u_0}(y_0,t_0).\]
Let $\gm:[0,t_0]\rightarrow M$ be a minimizer of $\bar{h}_{x_0,u_0}(x,t)$ with $\gm(0)=x_0$, $\gm(t_0)=y_0$. Define
\[F(s):=-h^{x_0,-u_0}(\gm(s),s)-\bar{h}_{x_0,u_0}(\gm(s),s)\quad s\in (0,t_0].\]
We add the definition of $F$ at $s=0$ by $F(0)=0$. By Implicit variational principle, we have
 \[\lim_{s\rightarrow 0^+}\bar{h}_{x_0,u_0}(\gm(s),s)=u_0.\]
Similar to \cite[Lemma 3.1 and Lemma 3.2]{WWY}, we have
\[\lim_{s\rightarrow 0^+}h^{x_0,-u_0}(\gm(s),s)=-u_0.\]
Moreover,  $F$ is continuous on $[0,t_0]$. Note that $F(t_0)>0$, one can find $s_0\in [0,t_0)$ such that $F(s_0)=0$ and $F(s)>0$ for each $s\in (s_0,t_0]$. By the definitions of $\bar{h}_{x_0,u_0}(x,t)$ and $h^{x_0,-u_0}(x,t)$, for each $s\in (s_0,t_0]$, there hold
\[\bar{h}_{x_0,u_0}(\gm(s),s)=\bar{h}_{x_0,u_0}(\gm(s_0),s_0)+\int_{s_0}^sL(\gm(\tau),-\bar{h}_{x_0,u_0}(\gm(\tau),\tau),-\dot{\gm}(\tau))d\tau,\]
\[-{h}^{x_0,-u_0}(\gm(s),s)\leq -\bar{h}^{x_0,-u_0}(\gm(s_0),s_0)+\int_{s_0}^sL(\gm(\tau),{h}^{x_0,-u_0}(\gm(\tau),\tau),-\dot{\gm}(\tau))d\tau,\]
which together with $|\frac{\partial H}{\partial u}|\leq \lambda$ implies
\[F(s)\leq \lambda\int_{s_0}^sF(\tau)d\tau.\]
By Gronwall inequality, we have $F(s)\equiv 0$ on $[s_0,t_0]$. This contradiction shows Proposition \ref{hupin} holds.
\end{proof}

Let $\check{S}_{y,u}$ be the set of $w\in C(M\times[0,+\infty),\R)$ satisfying $w(y,0)\leq u$ and $w_t+{H}(x,w,w_x)=0$ in the viscosity sense. Let $\hat{S}^{y,u}$ be the set of $v\in C(M\times[0,+\infty),\R)$ satisfying $v(y,0)\geq u$ and $-v$ is a solution of $v_t+{H}(x,-v,-v_x)=0$ in the viscosity sense. There holds
\begin{proposition}\label{repaction}
Given $y\in M$, $u\in \R$, for any $(x,t)\in M\times(0,+\infty)$,
\[{h}_{y,u}(x,t)=\sup\{w(x,t)\ |\ w\in \check{S}_{y,u}\},\]
\[{h}^{y,u}(x,t)=\inf\{v(x,t)\ |\ v\in \hat{S}^{y,u}\}.\]
\end{proposition}
\begin{proof}
We only need to prove the first equality, since the second one follows from the first one and Proposition \ref{hupin}. Let $w\in \check{S}_{y,u}$ and $\varphi(x):=w(x,0)$. By the monotonicity of $h_{x_0,u_0}(x,t)$ w.r.t. $u_0$,
	\begin{align*}
		w(x,t)=T_t^-\varphi(x)
=\inf_{z\in M}h_{z,\varphi(z)}(x,t)
\leq h_{y,\varphi(y)}(x,t)
\leq h_{y,u}(x,t),
		\end{align*}
which implies $h_{y,u}(x,t)\geq \sup\{w(x,t)\ |\ w\in \check{S}_{y,u}\}$.

On the other hand, we consider the following
\begin{equation}\label{chanww}
	\begin{cases}
w_t+{H}(x,w,w_x)=0,\quad (x,t)\in M\times (0,+\infty),\\
w(x,0)=u+\theta\phi(x), \quad x\in M,
\end{cases}
\end{equation}
where $\theta$ is a positive constant, $\phi\in C(M,\R)$ and $\phi(y)=0$, $\phi(x)>0$ for $x\neq y$. Let $w_\theta(x)$ be the viscosity solution of (\ref{chanww}). Then
\[w_\theta(x,t)=\inf_{z\in M}h_{z,u+\theta\phi(z)}(x,t).\]
We assert that for each $\delta>0$, there exists $\theta:=\theta(\delta)>0$ such that
\[w_\theta(x,t)=\inf_{d(z,y)<\delta}h_{z,u+\theta\phi(z)}(x,t).\]
 Note that $h_{x_0,u_0}(x,t)$ is Lipschitz continuous w.r.t. $x_0$ with Lipschitz constant $\iota:=\iota(u_0,x,t)$. If the assertion is true, for each $\eps>0$, there exists $\delta:=\frac{\eps}{\iota}$ such that
 	\begin{align*}
		h_{y,u}(x,t)&\leq \inf_{d(z,y)<\delta}h_{z,u}(x,t)+\eps\\
&\leq \inf_{d(z,y)<\delta}h_{z,u+\theta\phi(z)}(x,t)+\eps\\
&= w_\theta(x,t)+\eps,
		\end{align*}
which gives rise to
\[h_{y,u}(x,t)\leq \sup\{w_\theta(x,t)\ |\ \theta>0\}\leq \sup\{w(x,t)\ |\ w\in \check{S}_{y,u}\}.\]

It remains to prove the assertion. By contradiction, we assume there exists $\delta_0>0$, for each $\theta>0$, one can find $z_\theta\in M$ with $d(z_\theta, y)\geq \delta_0$ such that
\[w_\theta(x,t)=h_{z_\theta,u+\theta\phi(z_\theta)}(x,t).\]
Note that $\phi(z_\theta)\geq k_0>0$ since $d(z_\theta, y)\geq \delta_0$, it follows that
\[h_{y,u}(x,t)\geq w_\theta(x,t)\geq h_{z_\theta,u+\theta k_0}(x,t).\]
By the reversibility of $h_{x_0,u_0}(x,t)$ \cite[Proposition 3.4]{WWY1},  for each $A>0$, there exists $u_0\in \mathbb{R}$ such that $h_{x_0,u_0}(x,t)>A$ and $A$ is independent of $x_0$ due to the compactness of $M$. Hence, $h_{z_\theta,u+\theta k_0}(x,t)\rightarrow +\infty$ as $\theta\rightarrow +\infty$, which contradicts the boundedness of $h_{y,u}(x,t)$.
\end{proof}

\begin{proposition}\label{pr4.5}
Backward weak KAM solutions and  viscosity solutions  of  equation \eqref{hj} are the same. Moreover,
\begin{itemize}
\item
 [(i)] $u\in\mathcal{S}_-$ if and only if $T^-_tu=u$ for all $t\geq 0$;
 \item [(ii)] $u\in\mathcal{S}_+$ if and only if $T^+_tu=u$ for all $t\geq 0$.
 \end{itemize}
\end{proposition}
We will provide a proof of Proposition \ref{pr4.5} in Appendix A. In \cite{SWY}, it was shown that

\begin{proposition}\label{pr4.6}
	Assume (H1), (H2), (H3') and (S). For each $\varphi\in C(M,\mathbb{R})$, the uniform limit $\lim_{t\rightarrow +\infty}T^-_t\varphi(x)$ exists. Let $\varphi_\infty(x)=\lim_{t\rightarrow +\infty}T^-_t\varphi(x)$. Then $\varphi_\infty(x)=u_-(x)$ for all $x\in M$, where $u_-$ denotes the unique backward weak KAM solution of equation \eqref{hj}.
\end{proposition}

Under assumptions (H1), (H2), (H3') and (S), by Proposition \ref{pr4.6}, for any given $x_0\in M$, $u_0\in\mathbb{R}$ and $s>0$, we deduce that
\[
\lim_{t\rightarrow +\infty}h_{x_0,u_0}(x,t+s)=\lim_{t\rightarrow +\infty}T^-_th_{x_0,u_0}(x,s)
\]
exists. Thus, we can define a function on $M$ by
\[
h_{x_0,u_0}(x,+\infty):=\lim_{t\rightarrow +\infty}h_{x_0,u_0}(x,t),\quad x\in M.
\]
Moreover, we have
\begin{proposition}\label{pr4.655}
	Assume (H1), (H2), (H3') and (S).  For each $(x_0,u_0)\in M\times\R$, we have $h_{x_0,u_0}(x,+\infty)=u_-(x)$ for all $x\in M$, i.e.,  $h_{x_0,u_0}(x,+\infty)$ is the unique backward weak KAM solution of equation \eqref{hj}.
\end{proposition}
By \cite[Theorem 1.2]{WWY2}, there holds
\begin{proposition}\label{pairxx}
The uniform limit $\lim_{t\rightarrow +\infty}T^+_tu_-$ exists. Let $u_+=\lim_{t\rightarrow +\infty}T^+_tu_-$. Then
$u_+\in \mathcal{S}_+$, $u_-=\lim_{t\rightarrow +\infty}T^-_tu_+$ and $u_+$ is the maximal forward weak KAM solution.
\end{proposition}

For each $v_+\in \mathcal{S}_+$, we define
\[\mathcal{I}_{v_+}:=\{x\in M\ |\ u_{-}(x)=v_{+}(x)\}.\]
It was shown in \cite[Lemma 4.8]{WWY2} that both $u_-$ and $v_+$ are differentiable at $x\in\mathcal{I}_{v_+} $ and with the same derivative. Thus, one can define
\[
\tilde{\mathcal{I}}_{v_+}:=\{(x,u,p):\ x\in \mathcal{I}_{v_+},\ u=u_-(x)=v_+(x),\ p=Du_-(x)=Dv_+(x) \}.
\]
By \cite[Theorem 1.5]{WWY2}, we have
\begin{proposition}\label{omeggaa}
Given $x_0\in M$,
let $\eta:[0,+\infty)\rightarrow M$ be a $(v_+, L, 0)$-calibrated curve with $\eta(0)=x_0$. Let  $v_0:=v_+(x_0)$, $p_0:=\frac{\partial L}{\partial \dot{x}}(x_0,v_0,\dot{\eta}(0)_+)$, where $\dot{\eta}(0)_+$ denotes the right derivative of $\eta(t)$ at $t=0$. Let $\omega(x_0,v_0,p_0)$ be the $\omega$-limit set of $(x_0,v_0,p_0)$. Then
\[\omega(x_0,v_0,p_0)\subset \tilde{\mathcal{I}}_{v_+},\]
where  $\omega(x_0,u_0,p_0)$  denotes the $\omega$-limit  set for $(x_0,u_0,p_0)$.
\end{proposition}

\section{Proof of Main Result \ref{thmr1}}

%=======================================

\subsection{Proof of Main Proposition \ref{ii}}
As mentioned in the introduction, under assumptions (H1), (H2), (H3') and (S), equation \eqref{hj} has a unique backward weak KAM solution $u_-$. By Proposition \ref{pairxx}, there is at least one forward weak KAM solution $u_+:=\lim_{t\to+\infty}T^+_tu_-$ of equation \eqref{hj}, which implies that $\mathcal{S}_+\neq\emptyset$. So, we only need to show that if $\mathcal{S}_+\neq\emptyset$, then condition (S) holds true.

For any $v_+\in\mathcal{S}_+$,  define a subset of $T^*M\times\R$ associated with $v_+$ by
\[
G_{v_+}:=\mathrm{cl}\Big(\big\{(x,v,p): x\ \text{is a point of differentiability of}\ v_+,\  v=v_+(x),\ p=Dv_+(x)\big\}\Big),
\]
where $\mathrm{cl}(A)$ denotes the closure of $A\subset T^*M\times \R$. Let $\tilde{\Sigma}_{v_+}:=\bigcap_{t\geq 0}\Phi_{t}(G_{v_+})$.
Since $v_+$ is Lipschitz continuous \cite[Lemma 4.1]{WWY2}, then $G_{v_+}$ is well defined and it is a compact subset of $T^*M\times\R$. Let $\Phi_t$ denote the local flow of (\ref{c}) generated by $H(x,u,p)$. Recall that $G_{v_+}$ is invariant by $\Phi_t$ for each $t\geq 0$.
Note that for $s<0$, we have
\begin{align*}
\Phi_s(\tilde{\Sigma}_{v_+})=\Phi_s\left(\bigcap_{t\geq 0}\Phi_{t}(G_{v_+})\right)=\bigcap_{t\geq 0}\Phi_{t+s}(G_{v_+})\subset \bigcap_{t\geq 0}\Phi_{t}(G_{v_+})=\tilde{\Sigma}_{v_+}.
\end{align*}
So, it is a fact that $\tilde{\Sigma}_{v_+}$ is a non-empty, compact and  $\Phi_t$-invariant subset of $T^*M\times\R$.
Let $\Sigma_{v_+}:=\pi\tilde{\Sigma}_{v_+}$, where $\pi:T^*M\times\R\rightarrow M$ denotes the orthogonal projection.

To show condition (S) holds true, we proceed in three steps.

\medskip

\noindent {\bf Step 1}: For each $t\geq 0$, $T_t^-v_+\geq v_+$ everywhere.
\medskip

\noindent It is clear that $T_0^-v_+=v_+$. For $t>0$,  we have
\[
T^{-}_tv_+(x)=\inf_{y\in M}h_{y,v_+(y)}(x,t),\quad \forall x\in M.
\]
Thus, in order to prove $T_t^-v_+\geq v_+$ everywhere, it is sufficient to show that for each $y\in M$, $h_{y,v_+(y)}(x,t)\geq v_+(x)$ for all $(x,t)\in M\times (0,+\infty)$.
For any given $(x,t)\in M\times (0,+\infty)$, let $u(y):=h_{y,v_+(y)}(x,t)$ for all $y\in M$. Then $v_+(y)=h^{x,u(y)}(y,t)$. Since
\[
v_+(y)=T_t^+v_+(y)=\sup_{z\in M}h^{z,v_+(z)}(y,t),
\]
which implies $v_+(y)\geq h^{x,v_+(x)}(y,t)$, i.e., $h^{x,u(y)}(y,t)\geq h^{x,v_+(x)}(y,t)$. By the monotonicity of backward implicit action functions, we have $u(y)\geq v_+(x)$ for all $y\in M$, i.e., $h_{y,v_+(y)}(x,t)\geq v_+(x)$ for all $y\in M$.

\medskip
\noindent  {\bf Step 2}:  For each $t\geq 0$, $T_t^-v_+=v_+$ on $\Sigma_{v_+}$.
\medskip

\noindent By Step 1, we only need to prove $T_t^-v_+\leq v_+$ on $\Sigma_{v_+}$ for each $t>0$. For any $x\in\Sigma_{v_+}$, let $v:=v_+(x)$. Then there exists  $p\in T_x^*M$ such that $(x,v,p)\in \tilde{\Sigma}_{v_+}$. Fix $t>0$, let $(x(s),v(s),p(s)):=\Phi_{s-t}(x,v,p)$ with $(x(t),v(t),p(t))=(x,v,p)$ for $s\in\R$.  We assert that
\begin{equation}\label{huxuxu}
h_{x(s),v_+(x(s))}(x,t-s)=v,\quad \forall 0\leq s<t.
\end{equation}
If the assertion is true, then we have
\[
T_t^-v_+(x)=\inf_{y\in M}h_{y,v_+(y)}(x,t)\leq h_{x(0),v_+(x(0))}(x,t)=v=v_+(x).
\]
Now we prove assertion \eqref{huxuxu}.
The invariance of $\tilde{\Sigma}_{v_+}$ implies $v(s)=v_+(x(s))$ for all $s\in\R$. It remains to show $h_{x(s),v(s)}(x,t-s)=v$, equivalently, $h^{x,v}(x(s),t-s)=v(s)$. By the maximality of $h^{x,v}(x(s),t-s)$, we deduce that $h^{x,v}(x(s),t-s)\geq v(s)$. Assume by contradiction that there exists $s\in[0,t)$ such that $h^{x,v}(x(s),t-s)>v(s)$. Let $\gm:[0,t-s]\rightarrow M$ be a minimizer of $h^{x,v}(x(s),t-s)$ with $\gm(t-s)=x$ and $\gm(0)=x(s)$. Let $F(\sigma):=h^{x,v}(\gm(\sigma),t-s-\sigma)-v_+(\gm(\sigma))$, for $\sigma\in [0,t-s]$.  Since $F(\sigma)$ is continuous, $F(0)>0$ and $F(t-s)=0$, then one can find $s_0\in (0,t-s]$ such that $F(s_0)=0$ and $F(\sigma)>0$ for $\sigma\in [0,s_0)$. Note that
\begin{align*}
h^{x,v}(\gm(\sigma),t-s-\sigma)&=h^{x,v}(\gm(s_0),t-s-s_0)-\int^{s_0}_\sigma L(\gm(\tau),h^{x,v}(\gm(\tau),t-s-\tau),\dot{\gm}(\tau))d\tau,\\
v_+(\gm(s_0))&\leq v_+(\gm(\sigma))+\int^{s_0}_\sigma L(\gm(\tau),v_+(\gm(\tau)),\dot{\gm}(\tau))d\tau.
\end{align*}
It follows that
\[
F(\sigma)\leq \lambda\int^{s_0}_\sigma F(\tau)d\tau,
\]
which implies $F(\sigma)=0$ for all $\sigma\in [0,s_0)$. In particular, $F(0)=0$ in contradiction with $F(0)>0$.

\medskip
\noindent  {\bf Step 3}:
For $T_t^-v_+$, we have
\begin{itemize}
\item {\it Uniform boundedness}:  there exists a constant $K_1>0$ independent of $t$ such that for $t>1$,  \[\|T_t^-v_+\|_\infty\leq K_1;\]
\item {\it Equi-Lipschitz continuity}:  there exists a constant $\kappa_1>0$ independent of $t$ such that for $t> 2$, the function $x\mapsto T_t^-v_+(x)$ is $\kappa_1$-Lipschitz continuous on $M$.
\end{itemize}
\medskip
We prove the uniform boundedness first. By Step 1 and the compactness of $M$, $\{T_t^-v_+\}_{t\geq 0}$ is uniformly bounded from below.
On the other hand, for any given $y\in {\Sigma}_{v_+}$ and $t>1$, from Step 2 we get
\begin{align*}
T_t^-v_+(x)=T_1^-\circ T_{t-1}^-v_+(x)
=\inf_{z\in M}h_{z,T_{t-1}^-v_+(z)}(x,1)
\leq h_{y,T_{t-1}^-v_+(y)}(x,1)
= h_{y,v_+(y)}(x,1),
\end{align*}
which implies $\{T_t^-v_+\}_{t>1}$ is uniformly bounded form above. Denote by $K_1>0$ a constant  such that $\|T_t^-v_+\|_\infty\leq K_1$ for all $t>1$.

Then we prove the equi-Lipschitz continuity. Note that
\begin{align*}
|T_t^-v_+(x)-T_t^-v_+(y)|&=|\inf_{z\in M}h_{z,T_{t-1}^-v_+(z)}(x,1)-\inf_{z\in M}h_{z,T_{t-1}^-v_+(z)}(y,1)|\\
&\leq \sup_{z\in M}|h_{z,T_{t-1}^-v_+(z)}(x,1)-h_{z,T_{t-1}^-v_+(z)}(y,1)|.
\end{align*}
Since $h_{\cdot,\cdot}(\cdot,1)$ is Lipschitz on $M\times [-K_1,K_1]\times M$ with some Lipschitz constant $\kappa_1>0$.  It follows that
\[|T_t^-v_+(x)-T_t^-v_+(y)|\leq \kappa_1 d(x,y), \quad \forall t>2.\]

By Step 1 and Step 3, the uniform limit $\lim_{t\rightarrow+\infty}T_t^-v_+$ exists. Define
\[v_-:=\lim_{t\rightarrow+\infty}T_t^-v_+.
\]
It follows  that for any given $t\geq 0$,
\[
\|T_{t+s}^-v_+-T_t^-v_-\|_\infty\leq \|T_s^-v_+-v_-\|_\infty.
\]
Taking $s\rightarrow +\infty$, we have
$T_t^-v_-=v_-$ for all $t\geq 0$.
By Proposition \ref{pr4.5}, condition (S) holds true. This completes the proof of Main Proposition \ref{ii}.

%====================================
\subsection{Proof of Main Proposition \ref{iii}}

We divide the proof into two steps.

%\medskip
%\noindent Step 1:
%\[
%\tilde{\Sigma}_{v_+}:=\cap_{t\geq 0}\Phi_{t}(G_{v_+})=\tilde{\mathcal{I}}_{v_+}.
%\]
%\medskip
%By Step 2 in the proof of Main Result \ref{ii}, we have $\tilde{\Sigma}_{v_+}\subset \tilde{\mathcal{I}}_{v_+}$. It remains to prove
%$\tilde{\mathcal{I}}_{v_+}\subset \tilde{\Sigma}_{v_+}$.
% In fact, for each $(x_0,u_0,p_0)\in \tilde{\mathcal{I}}_{v_+}$, we have
%	\[(x(t),u(t),p(t))=\Phi_{t}(x_0,u_0,p_0)\in \tilde{\mathcal{I}}_{v_+}\subset G_{v_+}, \quad \forall t\in \R.\]
%	It follows that  $(x_0,u_0,p_0)\in \Phi_{t}(G_{v_+})$ for each $t\in \R$, and thus we have
%	\[(x_0,u_0,p_0)\in \cap_{t\geq 0}\Phi_{t}(G_{v_+})=\tilde{\Sigma}_{v_+}.\]

\medskip
\noindent  {\bf Step 1}:  There exists a constant $K_2>0$ such that $\|v_+\|_\infty\leq K_2$ for all $v_+\in\mathcal{S}_+$.
\medskip

\noindent By Proposition \ref{pairxx}, we get $v_+\leq u_+$ for all $v_+\in\mathcal{S}_+$. It is clear that $\{v_+\}_{v_+\in\mathcal{S}_+}$ is uniformly bounded from above. By Step 2 in the proof of Main Proposition \ref{ii}, for each $\bar{x}\in \Sigma_{v_+}$, we have $v_+(\bar{x})=u_-(\bar{x})$. Thus, for any given $\bar{x}\in \Sigma_{v_+}$, we have
\begin{align*}
v_+(x)=T_1^+v_+(x)=\sup_{z\in M}h^{z,v_+(z)}(x,1)\geq h^{\bar{x},v_+(\bar{x})}(x,1)= h^{\bar{x},u_-(\bar{x})}(x,1).
 \end{align*}
From the compactness of $M$, we deduce that $\{v_+\}_{v_+\in\mathcal{S}_+}$ is uniformly bounded from below. Denote by $K_2>0$ a constant  such that $\|v_+\|_\infty\leq K_2$ for all $v_+\in\mathcal{S}_+$.

\medskip
\noindent  {\bf Step 2}: There exists a constant $\kappa_2>0$ such that $v_+$ is $\kappa_2$-Lipschitz continuous on $M$ for all $v_+\in \mathcal{S}_+$.
\medskip

\noindent For each $x$, $y\in M$, let $\gamma:[0,d(x,y)]\to M$ be a geodesic of length $d(x,y)$, parameterized by arclength and connecting $x$ to $y$. Let
	\[
	\kappa_2:=\sup\{L(x,u,\dot{x})\ |\ x\in M,\ |u|\leq K_2,\ \|\dot{x}\|_x=1\}.
	\]
	Since $\|\dot{\gamma}(s)\|_{\gamma(s)}=1$ for all $s\in[0,d(x,y)]$ and $\|v_+\|_\infty\leq K_2$, we have
	 \[
	 L(\gamma(s),v_+(\gamma(s)),\dot{\gamma}(s))\leq \kappa_2,\quad \forall s\in[0,d(x,y)].
	 \]
	Since $v_+\prec L$, we have
	\begin{align*}
	v_+(y)-v_+(x)&\leq \int_0^{d(x,y)}L(\gamma(s),v_+(\gamma(s)),\dot{\gamma}(s))ds\leq\kappa_2\, d(x,y).
	\end{align*}
	We finish the proof  of Step 2 by exchanging the roles of $x$ and $y$.
{\subsection{Proof of Main Proposition \ref{vi}}
Before proving Main Result \ref{vi}, we need to show some preliminary results first. Note that in this paper $H$ satisfies (H1), (H2) and (H3'). Lemma \ref{hinequ}, Lemma \ref{nonex2} and Proposition \ref{nonex} below hold true under weaker conditions:  (H1), (H2) and $|\frac{\partial H}{\partial u}|\leq \lambda$ for all $(x,u,p)\in T^*M\times \mathbb{R}$.

\begin{lemma}\label{hinequ}
For any given $x\in M$ and $t>0$, $x_0\in M$, $u,v\in\R$, we have
\begin{equation*}
|h_{x_0,u}(x,t)-h_{x_0,v}(x,t)|\leq e^{\lambda t}|u-v|.
\end{equation*}
\end{lemma}
\begin{proof}
The monotonicity of $h_{x_0,u}(x,t)$ with respect to $u$ shows that if $u\geq v$, then $h_{x_0,u}(x,t)\geq h_{x_0,v}(x,t)$.

Since $u,v\in\R$, then we have the dichotomy: a) $u\leq v$, b) $u>v$. For Case a), we have $h_{x_0,u}(x,t)\leq h_{x_0,v}(x,t)$. Let $\gm_u$ be a minimizer of $h_{x_0,u}$ with $\gm_u(0)=x_0$ and $\gm_u(t)=x$. From the monotonicity of $h_{x_0,u}(x,t)$, it follows that for any $s\in (0,t]$,
\begin{equation}\label{tpro}
h_{x_0,u}(\gm_u(s),s)\leq h_{x_0,v}(\gm_u(s),s).
\end{equation}
In terms of the definition of $h_{x_0,u}(x,t)$, we have
\begin{align*}
&h_{x_0,v}(\gm_u(s),s)- h_{x_0,u}(\gm_u(s),s)\\
\leq &v-u+\int_0^sL(\gm_u(\tau),h_{x_0,v}(\gm_u(\tau),\tau),\dot{\gm}_u(\tau))
-L(\gm_u(\tau),h_{x_0,u}(\gm_u(\tau),\tau),\dot{\gm}_u(\tau))d\tau,\\
\leq & v-u+\int_0^s\lambda|h_{x_0,v}(\gm_u(\tau),\tau)- h_{x_0,u}(\gm_u(\tau),\tau)|d\tau.
\end{align*}
Let $F(\tau):=h_{x_0,v}(\gm_u(\tau),\tau)- h_{x_0,u}(\gm_u(\tau),\tau)$. It follows from (\ref{tpro}) that $F(\tau)\geq 0$ for any $\tau\in (0,t]$. Hence, we have
\[F(s)\leq v-u+\int_0^s\lambda F(\tau)d\tau.\]
By Gronwall's inequality, it yields
\begin{equation*}
F(s)\leq (v-u)e^{\lambda s}.
\end{equation*}
In particular, we verify Lemma \ref{hinequ} for Case a).

For Case b), we have $h_{x_0,u}(x,t)\geq h_{x_0,v}(x,t)$. Let $\gm_v$ be a minimizer of $h_{x_0,v}$ with $\gm_v(0)=x_0$ and $\gm_v(t)=x$. Let $G(\tau):=h_{x_0,u}(\gm_u(\tau),\tau)- h_{x_0,v}(\gm_u(\tau),\tau)$. By a similar argument as Case a), we have
\begin{equation*}
G(s)\leq (u-v)e^{\lambda s},
\end{equation*}
which completes the proof of the lemma.
\end{proof}

\begin{lemma}\label{nonex2}Given any $\varphi$, $\psi\in C(M,\mathbb{R})$,  we have
$\|T_t^-\varphi-T_t^-\psi\|_\infty\leq e^{\lambda t}\|\varphi-\psi\|_\infty$, \quad $\forall t\geq 0$.
\end{lemma}
\begin{proof}
Recall that for $t>0$ and each $x\in M$, we have
\begin{equation}\label{thphi}
T_t^-\varphi(x)=\inf_{y\in M}h_{y,\varphi(y)}(x,t),
\end{equation}
\begin{equation}\label{thpsi}
T_t^-\psi(x)=\inf_{z\in M}h_{z,\psi(z)}(x,t).
\end{equation}
Note that $h_{y,\varphi(y)}(x,t)$ is continuous with respect to $y$. Based on the compactness of $M$, the infimums in (\ref{thphi}) and (\ref{thpsi}) can be attained at $y_0$ and $z_0$ respectively. On one hand, by Lemma \ref{hinequ} we have
\begin{align*}
T_t^-\varphi(x)-T_t^-\psi(x)
\leq h_{z_0,\varphi(z_0)}(x,t)-h_{z_0,\psi(z_0)}(x,t)
\leq  e^{\lambda t}|\varphi(z_0)-\psi(z_0)|
\leq e^{\lambda t}\|\varphi-\psi\|_{\infty}.
\end{align*}
On the other hand, by Lemma \ref{hinequ} again we have
\begin{align*}
T_t^-\varphi(x)-T_t^-\psi(x)
\geq h_{y_0,\varphi(y_0)}(x,t)-h_{y_0,\psi(y_0)}(x,t)
\geq  -e^{\lambda t}|\varphi(y_0)-\psi(y_0)|
\geq -e^{\lambda t}\|\varphi-\psi\|_{\infty}.
\end{align*}
Hence,
\[\|T_t^-\varphi(x)-T_t^-\psi(x)\|_{\infty}\leq e^{\lambda t}\|\varphi-\psi\|_{\infty}.\]
This completes the proof of the lemma.
\end{proof}

\begin{proposition}\label{nonex}Given any $\varphi$, $\psi\in C(M,\mathbb{R})$,  we have
$\|T_t^+\varphi-T_t^+\psi\|_\infty\leq e^{\lambda t}\|\varphi-\psi\|_\infty$, $\forall t\geq 0$.
\end{proposition}

\begin{proof}

 It is clear that the proposition holds for $t=0$. Note that $\bar{H}(x,u,p)$ satisfies (H1), (H2) and $|\frac{\partial \bar{H}}{\partial u}|\leq \lambda$. By \eqref{sg},
for each $\varphi\in C(M,\R)$, there holds
\begin{align*}
-T^+_t(-\varphi)=\bar{T}^-_t\varphi,\quad \forall t\geq 0,
\end{align*}
where $\bar{T}^-_t$ denotes the backward solution semigroup associated to $\bar{H}$. For $t>0$, by Lemma \ref{nonex2} we have
\[\|T_t^+\varphi-T_t^+\psi\|_\infty=\|-\bar{T}_t^-(-\varphi)+\bar{T}_t^-(-\psi)\|_\infty\leq e^{\lambda t}\|\varphi-\psi\|_\infty.\]
	
\end{proof}

Under the assumptions (H1), (H2) and (H3'), we will show the existence of minimal forward weak KAM solutions of equation \eqref{hj}. Note that $(\mathcal{S}_+,\leq)$ is a partially ordered set. In view of  Zorn's lemma, if every chain in $\mathcal{S}_+$ has a lower bound in $\mathcal{S}_+$, then $\mathcal{S}_+$ contains at least one minimal element.
So, in order to prove Main Proposition \ref{vi}, it suffices to prove the following result.
\begin{proposition}\label{ttorder}
Let $A$ be a totally ordered subset of $\mathcal{S}_+$. Let $\bar{u}(x):=\inf_{u\in A}u(x)$ for each $x\in M$. Then $\bar{u}\in \mathcal{S}_+$.
\end{proposition}

Before proving the proposition, we show the following lemma.
\begin{lemma}\label{f_un}
Let $\{u_n\}_{n\in \mathbb{N}}\subset \mathcal{S}_+$. If there exists $\bar{u}$ such that $u_n$ converges to $\bar{u}$ pointwise on $M$, then $u_n$ converges to $\bar{u}$ uniformly on $M$. Moreover, $\bar{u}\in \mathcal{S}_+$.
\end{lemma}

\begin{proof}
By Main Proposition \ref{iii}, $\{u_n\}_{n\in\mathbb{N}}$ is equi-Lipschitz on $M$. Thus, we deduce that $u_n$ converges uniformly to $\bar{u}$ as $n\to+\infty$.
To show $\bar{u}\in \mathcal{S}_+$, it suffices to verify $T_t^+\bar{u}=\bar{u}$ for each $t\geq 0$. By Proposition \ref{nonex},
\[\|T_t^+u_n-T_t^+\bar{u}\|_\infty\leq e^{\lambda t}\|u_n-\bar{u}\|_\infty,\]
which means that for each $t\geq 0$, there holds
\[\lim_{n\rightarrow\infty}T_t^+u_n=T_t^+\bar{u}.\]
Hence, $T_t^+\bar{u}=\bar{u}$ follows from $T_t^+u_n=u_n$. This completes the proof of the lemma.
\end{proof}

\begin{proof}[Proof of Proposition \ref{ttorder}]
The result holds true if $A$ is a finite set. We consider the case that $A$ has infinite elements. According to Lemma \ref{f_un}, it suffices to show that there exists a sequence $\{u_n\}_{n\in \mathbb{N}}\subset A$  such that  $u_n\rightarrow\bar{u}$ pointwise on $M$ as $n\to+\infty$.

By Main Proposition \ref{iii}, all $u\in \mathcal{S}_+$ are $\kappa$-equi-Lipschitz continuous. Note that $A\subset \mathcal{S}_+$, we have
\[|\bar{u}(x)-\bar{u}(y)|\leq \sup_{u\in\mathcal{S}_+}|u(x)-u(y)|\leq \kappa |x-y|.\]
It follows from the compactness that $M$ is separable. Namely, one can find a countable dense subset denoted by $U:=\{x_1,x_2,\ldots,x_n,\ldots\}$. We assert that there exist $\{u_n\}_{n\in \mathbb{N}}\subset A$ such that  for a given $n\in \mathbb{N}$ and each $i\in \{1,2,\ldots,n\}$,
\[0\leq u_n(x_i)-\bar{u}(x_i)<\frac{1}{n}.\]
Based on this assertion, one can conclude that $u_n\rightarrow\bar{u}$ pointwise on $M$ as $n\to+\infty$. In fact, for each $x\in M$, there exists a subsequence $V:=\{x_m\}_{m\in \mathbb{N}}\subset U$ such that $|x_m-x|<\frac{1}{m}$. Given $x\in M$ and $n\in \mathbb{N}$, for  each $x_i\in \{x_1,x_2,\ldots,x_n\}\cap V$, up to a rearrangement of $i$, we have
\begin{align*}
|u_n(x)-\bar{u}(x)|&\leq |u_n(x)-u_n(x_i)|+|u_n(x_i)-\bar{u}(x_i)|+|\bar{u}(x)-\bar{u}(x_i)|\\
&\leq 2\kappa |x_i-x|+\frac{1}{n}\\
&\leq \frac{2\kappa}{i}+\frac{1}{n}.
\end{align*}
Let $n\rightarrow \infty$ and $i\rightarrow\infty$ successively. It follows that $u_n\rightarrow\bar{u}$ pointwise on $M$.

It remains to prove the assertion. By the definition of $\bar{u}$, we have $u_n\geq \bar{u}$ on $M$. In the following, we construct $\{u_n\}_{n\in \mathbb{N}}\subset A$ such that  for a given $n\in \mathbb{N}$ and each $i\in \{1,2,\ldots,n\}$,
\[u_n(x_i)-\bar{u}(x_i)<\frac{1}{n}.\]
First of all, for $x_1\in U$, one can choose $v_1\in A$ such that
$v_1(x_1)-\bar{u}(x_1)<1/n$.
Let $x_j\in U$ with $j\leq n$ be the far left point such that $v_1(x_j)-\bar{u}(x_j)\geq 1/n$. Namely, for each $i\leq j-1$, there hold $v_1(x_i)-\bar{u}(x_i)<1/n$. For $x_j$, one can choose $v_2\in A$ such that
$v_2(x_j)-\bar{u}(x_j)<1/n$. Then we have
\[v_2(x_j)<\bar{u}(x_j)+\frac{1}{n}\leq v_1(x_j).\]
Note that $A$ is totally ordered, it yields $v_2\leq v_1$ on $M$. Moreover, for each $i\leq j-1$,
\[v_2(x_i)-\bar{u}(x_i)\leq v_1(x_i)-\bar{u}(x_i)<\frac{1}{n}.\]
Namely, we find an element  $v_2\in A$ such that for each $i\in \{1,\ldots, j\}$, $v_2(x_i)-\bar{u}(x_i)<1/n$. Replace $v_1$ by $v_2$ and repeat the process above, it follows that  there exists $v_k\in A$  with $k\leq n$ such that each $i\in \{1,2,\ldots,n\}$,
\[v_k(x_i)-\bar{u}(x_i)<\frac{1}{n}.\]
Take $u_n=v_k$, then $u_n\in A\subset \mathcal{S}_+$. This completes the proof.
\end{proof}

\medskip

\subsection{Proof of Main Proposition \ref{345}}
\begin{lemma}\label{ooo}
Let $v_+\in \mathcal{S}_+$. Given $x_0\in M$, let $\gm:[0,+\infty)\rightarrow M$ be a $(v_+,L,0)$-calibrated curve with $\gm(0)=x_0$. Let $v_0:=v_+(x_0)$, $p_0:=\frac{\partial L}{\partial \dot{x}}(x_0,v_0,\dot{\gm}(0)_+)$ and $(x(t),v(t),p(t)):=\Phi_t(x_0,v_0,p_0)$ for $t\geq 0$. Then we have $x(t)=\gm(t)$ and $v(t)=v_+(x(t))$ for all $t\geq 0$, and for each $t_2>t_1\geq 0$, there holds
\[
v(t_1)=h^{x(t_2),v(t_2)}(x(t_1),t_2-t_1).
\]
\end{lemma}
\begin{proof}
By similar arguments used in the proof of \cite[Proposition 4.1]{WWY2}, it is not difficult to show that $x(t)=\gm(t)$ and $v(t)=v_+(x(t))$ for all $t\geq 0$.
  By the maximality of $h^{x(t_2),v(t_2)}(x(t_1),t_2-t_1)$, we have
\[
v(t_1)\leq h^{x(t_2),v(t_2)}(x(t_1),t_2-t_1).
\]
On the other hand, since $T_t^+v_+=v_+$ for all $t\geq 0$, we have
\begin{align*}
v(t_1)&=v_+(x(t_1))=T_{t_2-t_1}^+v_+(x(t_1))
=\sup_{y\in M}h^{y,v_+(y)}(x(t_1),t_2-t_1)
\geq  h^{x(t_2),v(t_2)}(x(t_1),t_2-t_1).
\end{align*}
This completes the proof.
\end{proof}
\begin{proof}[Proof of Main Proposition \ref{four2}] Given $\eps>0$, denote by
	\[
	\mathcal{B}_\eps(S):=\left\{x\in M\ |\ d(x,S)<\eps\right\}
	\]
	the $\eps$-neighborhood of $S\subset M$,  where $d(\cdot,\cdot)$ denotes the distance function defined by the Riemannian metric on $M$. Then we have
	\[
	m(x,\xi)=\lim_{\eps\rightarrow 0^+}\sup_{y\in \mathcal{B}_\eps(\xi)}\sup_{\tau>0}h^{y,v_+(y)}(x,\tau),\quad x\in M.
	\]

	Given any $v_+\in\mathcal{S}_+$, for any $\xi\in\mathcal{I}_{v_+}$, it is straightforward to see that $m(x,\xi)$ is well defined.
	
	We first show that for each $v_+\in \mathcal{S}_+$ and each $\xi\in \mathcal{I}_{v_+}$, we have $m(x,\xi)\leq v_+(x)$ for each $x\in M$.  In fact,
	\[
	v_+(x)=T_t^+v_+(x)=\sup_{y\in M}h^{y,v_+(y)}(x,t), \quad \forall (x,t)\in M\times(0,+\infty),
	\]
	which implies for each $y\in M$,
	\[
	v_+(x)\geq \sup_{\tau>0}h^{y,v_+(y)}(x,\tau)\geq m(x,\xi),\quad \forall x\in M,
	\]
	
	Next  we show that for any $x_0\in M$, there exists $\bar{\xi}:=\bar{\xi}(x_0)\in \mathcal{I}_{v_+}$ such that $v_+(x_0)\leq m(x_0,\bar{\xi})$, which together with $m(x,\xi)\leq v_+(x)$ for each $x\in M$, implies that
	\[
	v_+(x)=\sup_{\xi\in \mathcal{I}_{v_+}}m(x,\xi),\quad \forall x\in M.
	\]
	Given $x_0\in M$,  let $\gm:[0,+\infty)\rightarrow M$ be a $(v_+,L,0)$-calibrated curve with $\gm(0)=x_0$. Let $v_0:=v_+(x_0)$, $p_0:=\frac{\partial L}{\partial \dot{x}}(x_0,v_0,\dot{\gm}(0)_+)$ and $(x(t),v(t),p(t)):=\Phi_t(x_0,v_0,p_0)$ for $t\geq 0$.  Recall that $\omega(x_0,v_0,p_0)$ denotes the $\omega$-limit set for $(x_0,v_0,p_0)$. For each $(\bar{\xi},\bar{v},\bar{p})\in \omega(x_0,v_0,p_0)$, there exists a sequence $\{t_n\}_{n\in\mathbb{N}}$ with $t_n\rightarrow+\infty$ as $n\rightarrow+\infty$, such that $\xi_n:=x(t_n)\rightarrow \bar{\xi}$ as $n\rightarrow+\infty$.
	By Lemma \ref{ooo}, we have
	\[v_+(x_0)=v_0=h^{x(t_n),v(t_n)}(x_0,t_n)=h^{\xi_n,v_+(\xi_n)}(x_0,t_n)\leq \sup_{\tau>0}h^{\xi_n,v_+(\xi_n)}(x_0,\tau).\]
	By Proposition \ref{omeggaa}, we get $\bar{\xi}\in \mathcal{I}_{v_+}$. Since $\xi_n\rightarrow \bar{\xi}$ as $n\rightarrow+\infty$, then  for any given $\eps>0$, there is $N\in\mathbb{N}$ such that $\xi_n\in \mathcal{B}_\eps(\mathcal{I}_{v_+})$ for $n>N$. Thus, for $n>N$, we have
	\[
	v_+(x_0)\leq \sup_{\tau>0}h^{\xi_n,v_+(\xi_n)}(x_0,\tau)\leq \sup_{y\in \mathcal{B}_\eps(\bar{\xi})}\sup_{\tau>0}h^{y,v_+(y)}(x_0,\tau).
	\]
	Letting $\eps\to0^+$, we have $v_+(x_0)\leq m(x_0,\bar{\xi})$. This completes the proof of Main Proposition \ref{four2}.
	
\end{proof}

\medskip

\subsection{Proof of Main Proposition \ref{3}}
The fact that $v_+$ and $u_-$ are of class $C^{1,1}$ on $\mathcal{I}_{v_+}$ is a direct consequence of  \cite[Proposition 4.2]{WWY2}.
In particular,  $u_-$ and $u_+$ are of class $C^{1,1}$ on $\mathcal{I}_{u_+}$.
\medskip

\medskip

\subsection{Proof of Main Proposition \ref{345}} For any $v_+\in\mathcal{S}_+$,
 we have
\[
\tilde{\mathcal{I}}_{v_+}\subset \tilde{\mathcal{I}}_{u_+},
\]
which implies $\mathcal{I}_{v_+}\subset \mathcal{I}_{u_+}$. Thus, for each $x\in \mathcal{I}_{v_+}$, we have
$u_-(x)=v_+(x)=u_+(x).$ If $v'_+\in\mathcal{S}_+$ with $v_+\leq v'_+$ everywhere,
then we get $v_+\leq v'_+\leq u_+$ everywhere. It gives rise to
$v'_+(x)=u_+(x)=u_-(x)$ for any $x\in \mathcal{I}_{v_+}$,
which implies Main Proposition \ref{345}(1).

For each $x\in M$, let $\gamma:[0,+\infty)\to M$ be a $(v'_+,L,0)$-calibrated curve with $\gm(0)=x$. By Lemma \ref{ooo}, we have
\[
v'_+(x)=h^{\gamma(t),v'_+(\gamma(t))}(x,t)\]
for all $t>0$.  By Proposition \ref{omeggaa}, there exists $t_0>0$ such that $\gamma(t_0)\in \mathcal{O}$. Thus, we deduce that
\[
v''_+(x)\geq h^{\gamma(t_0),v''_+(\gamma(t_0))}(x,t_0)\geq h^{\gamma(t_0),v'_+(\gamma(t_0))}(x,t_0)=v'_+(x).
\]
Item (3) follows from Item (2) directly. This completes the proof.

\medskip

%=========================================

\section{Proof of Main Result \ref{thmr2}}

\subsection{Proof of Main Proposition \ref{v}}

In order to prove Main Proposition \ref{v}, we show the following lemma first.
\begin{lemma}\label{keiiyy}
Given $x_0\in M$, $v_0\in \R$, $h^{x_0,v_0}(\cdot,\cdot)$ is bounded on $M\times [\delta,+\infty)$ for any given $\delta>0$ if and only if $v_0=u_-(x_0)$. More precisely, there hold
\begin{itemize}
\item [(i)] if $v_0=u_-(x_0)$, then $h^{x_0,v_0}(\cdot,\cdot)$ is bounded on $M\times [\delta,+\infty)$ for any $\delta>0$;
\item [(ii)] if $v_0>u_-(x_0)$, then $\lim_{t\to+\infty}h^{x_0,v_0}(x,t)=+\infty$ uniformly on $x\in M$;
\item [(iii)] if $v_0<u_-(x_0)$, then $\lim_{t\to+\infty}h^{x_0,v_0}(x,t)=-\infty$ uniformly on $x\in M$.
\end{itemize}
\end{lemma}
\begin{proof}
By \cite[Lemma 4.5]{WWY2}, for each $t\geq 0$, we have $T^+_tu_-\leq u_-$ everywhere.

\medskip

\noindent  {\bf Case (i)}:  For any $(x,t)\in M\times (0,+\infty)$, we get
\[h^{x_0,v_0}(x,t)=h^{x_0,u_-(x_0)}(x,t)\leq T_t^+u_-(x)\leq u_-(x).\]
So, $h^{x_0,v_0}(\cdot,\cdot)$ is bounded from above on $M\times (0,+\infty)$.
Let $\gm:(-\infty,0]\rightarrow M$ be a $(u_-,L,0)$-calibrated curve with $\gm(0)=x_0$. Let $p_0:=\frac{\partial L}{\partial \dot{x}}(x_0,v_0,\dot{\gm}(0)_-)$. Define $(x(-t),u(-t),p(-t)):=\Phi_{-t}(x_0,v_0,p_0)$ for $t\geq 0$.
In view of \cite[Proposition 4.1]{WWY2}, we have $u(-t)=u_-(x(-t))$. We assert that
\begin{equation}\label{jji}
u(-t)=h^{x_0,v_0}(x(-t),t),\quad \forall t>0.
\end{equation}
In fact, by the maximality of $h^{x_0,v_0}(x(-t),t)$, we have $u(-t)\leq h^{x_0,v_0}(x(-t),t)$ for any $t>0$. On the other hand,
\[
u(-t)=u_-(x(-t))\geq T_{t}^+u_-(x(-t))\geq h^{x_0,u_-(x_0)}(x(-t),t)=h^{x_0,v_0}(x(-t),t),\quad \forall t>0.
\]
Hence, assertion (\ref{jji}) is true. By Markov property of $h^{x_0,u_-(x_0)}(x,t)$ and \eqref{jji}, for any $t\geq \delta$, we have
\begin{align*}
h^{x_0,v_0}(x,t)&=h^{x_0,u_-(x_0)}(x,t)\\&=\sup_{y\in M}h^{y,h^{x_0,u_-(x_0)}(y,t-\frac{\delta}{2})}(x,\frac{\delta}{2})\\
&\geq h^{x(-(t-\frac{\delta}{2})),h^{x_0,u_-(x_0)}(x(-(t-\frac{\delta}{2})),t-\frac{\delta}{2})}(x,\frac{\delta}{2})\\
&=h^{x(-(t-\frac{\delta}{2})),u(-(t-\frac{\delta}{2}))}(x,\frac{\delta}{2}).
\end{align*}
Note that $u\left(-(t-\frac{\delta}{2})\right)=u_-\left(x(-(t-\frac{\delta}{2}))\right)$ is bounded on $[\delta,+\infty)$. Since $h^{\cdot,\cdot}(\cdot,\frac{\delta}{2})$ is locally Lipschitz on $M\times \R\times M$, then $h^{x_0,v_0}(\cdot,\cdot)$ is bounded from below on $M\times [\delta,+\infty)$.

\medskip

\noindent  {\bf Case (ii)}:
%By the monotonicity of $h^{x_0,v_0}(x,t)$ with respect to $v_0$, it follows from $v_0>u_-(x_0)$ that $h^{x_0,v_0}(x,t)>h^{x_0,u_-(x_0)}(x,t)$. Case (i) shows that $h^{x_0,v_0}(\cdot,\cdot)$ is bounded from below on $M\times [\delta,+\infty)$.
Assume by contradiction that there exists $C_1>0$ and a sequence $\{(x_n,t_n)\}_{n\in \mathbb{N}}\in M\times (0,+\infty)$ with $t_n\rightarrow +\infty$ as $n\rightarrow+\infty$ such that
$|h^{x_0,v_0}(x_n,t_n)|\leq C_1$. Let $v_n:=h^{x_0,v_0}(x_n,t_n)$ for all $n\in \mathbb{N}$. Then $h_{x_n,v_n}(x_0,t_n)=v_0$ for all $n\in \mathbb{N}$. Passing to a subsequence if necessary, we may suppose that
\[
x_n\rightarrow \bar{x},\quad v_n\rightarrow \bar{v},\quad \text{as}\ n\rightarrow +\infty.
\]
By Proposition \ref{pr4.655}, $\lim_{t\rightarrow+\infty}h_{\bar{x},\bar{v}}(x_0,t)=u_-(x_0)$. In particular,
\begin{equation}\label{iijjii}
\lim_{n\rightarrow+\infty}h_{\bar{x},\bar{v}}(x_0,t_n)=u_-(x_0).
\end{equation}

We assert that there exists a constant $C_2>0$ independent of $n$ such that
 \begin{align}\label{18053}
 \left|v_0-h_{\bar{x},\bar{v}}(x_0,t_n)\right|=\left|h_{x_n,v_n}(x_0,t_n)-h_{\bar{x},\bar{v}}(x_0,t_n)\right|\leq C_2\left(d(x_n,\bar{x})+|v_n-\bar{v}|\right).
 \end{align}
If the assertion is true, then $\lim_{n\rightarrow+\infty}h_{\bar{x},\bar{v}}(x_0,t_n)=v_0>u_-(x_0)$ in contradiction with (\ref{iijjii}).
So, we only need to prove the assertion.
Let $\gm_n:[0,t_n]\to M$ be a minimizer of $h_{x_n,v_n}(x_0,t_n)$. Define $u_n(s):=h_{x_n,v_n}(\gm_n(s),s)$ for $s\in[0,t_n]$. Let
\[
u_{n,1}:=u_n(1),\quad y_{n,1}:=\gm_n(1),\quad \bar{u}_n:=h^{y_{n,1},u_{n,1}}(\bar{x},1).
\]
Then $u_{n,1}=h_{x_n,v_n}(y_{n,1},1)$, or equivalently, $v_n=h^{y_{n,1},u_{n,1}}(x_n,1)$. Note that $v_n\to \bar{v}$ as $n\to+\infty$. Thus, by the local Lipschitz property of $h_{\cdot,\cdot}(\cdot,1)$ and the compactness of $M$, there is a constant $C_3>0$ such that $|u_{n,1}|\leq C_3$ for all $n\in \mathbb{N}$. Note that $h^{\cdot,\cdot}(\cdot,1)$ is Lipschitz on $M\times[-C_3,C_3]\times M$ with a Lipschitz constant $C_4>0$. So, we have
\begin{align}\label{18052}
|v_n-\bar{u}_n|=\left|h^{y_{n,1},u_{n,1}}(x_n,1)-h^{y_{n,1},u_{n,1}}(\bar{x},1)\right|\leq C_4\ d(x_n,\bar{x}).
\end{align}
By the Markov property of forward implicit action functions, the definitions of $u_{n,1}$ and $\bar{u}_n$,  we have
\begin{align}\label{1805}
h_{\bar{x},\bar{u}_n}(x_0,t_n)\leq h_{y_{n,1},h_{\bar{x},\bar{u}_n}(y_{n,1},1)}(x_0,t_n-1)=h_{y_{n,1},u_{n,1}}(x_0,t_n-1)=h_{x_n,v_n}(x_0,t_n).
\end{align}
By the monotonicity of the forward implicit actions and \eqref{18052}, we get
\begin{align*}\label{18051}
	h_{\bar{x},\bar{v}}(x_0,t_n)-h_{\bar{x},\bar{u}_n}(x_0,t_n)\leq|\bar{v}-\bar{u}_n|\leq |\bar{v}-v_n|+|v_n-\bar{u}_n|\leq |\bar{v}-v_n|+C_4\ d(x_n,\bar{x}),
\end{align*}
which together with \eqref{1805}, implies that
\[
h_{\bar{x},\bar{v}}(x_0,t_n)\leq h_{x_n,v_n}(x_0,t_n)+ |\bar{v}-v_n|+C_4\ d(x_n,\bar{x}).
\]
Similarly, one can show that
\[
h_{x_n,v_n}(x_0,t_n)\leq h_{\bar{x},\bar{v}}(x_0,t_n)+ |\bar{v}-v_n|+C_5\ d(x_n,\bar{x})
\]
for some constant $C_5>0$ independent of $n$. Hence, assertion \eqref{18053} holds true.

\medskip

\noindent  {\bf Case (iii)}: By a similar argument used in Case (ii) we can show (iii).
\end{proof}

\medskip

\begin{proof}[Proof of  Main Proposition \ref{v}]
Item (3) in Main Proposition \ref{v} is an easy consequence of Item (2) in  Main Proposition \ref{77iii}.
We will add some lines to show Item (3) in next subsection. See Remark \ref{add} below.

Here, we prove Items (1), (2).  First of all, we  show that if $\varphi$ satisfies (a') and (b'), then the family $\{T^+_t\varphi\}_{t\geq 0}$ is uniformly bounded on $M$.
\medskip

Since $\varphi\leq u_-$ everywhere, by \cite[Lemma 4.4]{WWY2}, then for each $t\geq 0$, $T_t^+\varphi\leq T_t^+u_-\leq u_-$ everywhere. On the other hand, for each $t>0$, \[T_t^+\varphi(x)\geq h^{x_0,\varphi(x_0)}(x,t)=h^{x_0,u_-(x_0)}(x,t).\]
By Lemma \ref{keiiyy}, we deduce that the function $(x,t)\mapsto T_t^+\varphi(x)$ is bounded from below by some constant $C_6$ on $M\times [1,+\infty)$.
It is clear that the function $(x,t)\mapsto T_t^+\varphi(x)$ is bounded from below by some constant $C_7$ on $M\times [0,1]$. Hence, the family $\{T^+_t\varphi\}_{t\geq 0}$ is uniformly bounded on $M$.

Next, we show that $\varphi$ satisfies (a') and (b'), provided the family $\{T^+_t\varphi\}_{t\geq 0}$ is uniformly bounded on $M$. Suppose not. It is now convenient to distinguish two cases.

\medskip

\noindent  { Case (1)}: If $\varphi(x_0)>u_-(x_0)$
for some $x_0\in M$, then from $T_t^+\varphi(x)\geq h^{x_0,\varphi(x_0)}(x,t)$ and  Lemma \ref{keiiyy}, we deduce that for each $x\in M$, $T_t^+\varphi(x)\rightarrow+\infty$ as $t\rightarrow+\infty$, a contradiction.

\medskip

\noindent  { Case (2)}: If $\varphi(y)<u_-(y)$ for all $y\in M$, it follows from Lemma \ref{keiiyy} that for all $x,y\in M$, $h^{y,\varphi(y)}(x,t)\rightarrow-\infty$ as $t\rightarrow+\infty$. Let $k_0:=\min_{x,y\in M}h^{y,\varphi(y)}(x,1)$. For any given $k<k_0$, define
\[
\sigma_k(x,y):=\max\left\{t\ |\ h^{y,\varphi(y)}(x,t)\geq k\right\},\quad x,\, y\in M.
\]
It is clear that $\sigma_k(\cdot,\cdot)$ is continuous on $M\times M$. Since $M$ is compact, then
\[\bar{\sigma}_k:=\max_{(x,y)\in M\times M}\sigma_k(x,y)\] is well defined. It follows that for each $t>\bar{\sigma}_k$, $h^{y,\varphi(y)}(x,t)\leq k$ for all $x$, $y\in M$. Then
\[
T_t^+\varphi(x)=\sup_{y\in M}h^{y,\varphi(y)}(x,t)\leq k, \quad \forall x\in M,
\]
which implies that for each $x\in M$, $T_t^+\varphi(x)\rightarrow-\infty$ as $t\rightarrow+\infty$, a contradiction.
\end{proof}

\subsection{Proof of  Main Proposition \ref{77iii}}

%In order to finish the proof of Item (1), it remains to show that there is $K>0$ such that for each
%$\varphi$ satisfying (a') and (b'), there is $T_{\varphi}>0$ such that
%\begin{align}\label{bc}
%|T^+_t\varphi(x)|\leq K,\quad \forall (x,t)\in M\times [0,+\infty).
%\end{align}
%We will see that \eqref{bc} is an easy consequence of Step 2 in Item (3) below. See details at the end of Step 2 below.
%\medskip
\medskip
\noindent  {\bf Sufficient condition (1)}: Note that $T_t^+u_+=u_+$ and
\[T_t^+u_+\leq T_t^+\varphi\leq T_t^+u_-.\]
Since  $u_+=\lim_{t\rightarrow \infty}T_t^+u_-$. It follows that $\lim_{t\rightarrow \infty}T_t^+\varphi=u_+$.
\medskip

\medskip
\noindent  {\bf Sufficient condition (2)}: We divide the proof into 2 steps.
\medskip

\noindent {\bf Step 1}:
For each $\delta>0$, we show that there is  $\kappa_\delta>0$ such that for each initial data $\varphi$ satisfying (a') and (b'),  the function $x\mapsto T^+_t\varphi(x)$ is $\kappa_\delta$-Lipschitz on $M$ for each $t\geq \delta$.

\medskip

By Item (1) in Main Proposition \ref{v}, there is a constant $K>0$ such that $|T^+_t\varphi(x)|\leq K$ for all $(x,t)\in M\times [0,+\infty)$.
Note that for any $t\geq\delta$, we have
\begin{align*}
\left|T_t^+\varphi(x)-T_t^+\varphi(y)\right|&=\left|\sup_{z\in M}h^{z,T_{t-\frac{\delta}{2}}^+\varphi(z)}(x,\frac{\delta}{2})-\sup_{z\in M}h^{z,T_{t-\frac{\delta}{2}}^+\varphi(z)}(y,\frac{\delta}{2})\right|\\
&\leq \sup_{z\in M}\left|h^{z,T_{t-\frac{\delta}{2}}^+\varphi(z)}(x,\frac{\delta}{2})-h^{z,T_{t-\frac{\delta}{2}}^+\varphi(z)}(y,\frac{\delta}{2})\right|.
\end{align*}
Since $h^{\cdot,\cdot}(\cdot,\frac{\delta}{2})$ is Lipschitz on $M\times [-K,K]\times M$ with some Lipschitz constant $\kappa_\delta$, then
\[
\left|T_t^+\varphi(x)-T_t^+\varphi(y)\right|\leq \kappa_\delta\ d(x,y), \quad \forall t\geq\delta.
\]

\medskip

\noindent {\bf Step 2}: We  show that for  $\varphi$ satisfying (a') and (b'), both
\begin{equation}\label{88iii}
\hat{\varphi}(x):=\limsup_{t\rightarrow+\infty}T_t^+\varphi(x),\quad \check{\varphi}(x):=\liminf_{t\rightarrow+\infty}T_t^+\varphi(x)
\end{equation}
are forward weak KAM solutions of (\ref{hj}).
It is clear to see that if (\ref{88iii}) is true, then Sufficient condition (2) holds.

First, we show $\hat{\varphi}(x)$ is a forward weak KAM solution of (\ref{hj}), i.e.,
$\hat{\varphi}$ satisfies $T_t^+\hat{\varphi}=\hat{\varphi}$ for any $t>0$.
From the uniform boundedness of $\{T^+_t\varphi\}_{t\geq 0}$, it is clear that $\hat{\varphi}(x)$ is well-defined. By definition, we have
\[
\lim_{t\to+\infty}\sup_{s\geq t}T^+_s\varphi(x)=\hat{\varphi}(x),\quad \forall x\in M.
\]
Since
\[
|\sup_{s\geq t}T^+_s\varphi(x)-\sup_{s\geq t}T^+_s\varphi(y)|\leq\sup_{s\geq t}|T^+_s\varphi(x)-T^+_s\varphi(y)|\leq \kappa_\delta d(x,y),
\quad \forall t>\delta>0,
\]
then
\begin{align}\label{5-230}
\lim_{t\to+\infty}\sup_{s\geq t}T^+_s\varphi(x)=\hat{\varphi}(x)
\end{align}
uniformly on $x\in M$.

 It is clear that $T_0^+\varphi=\varphi$. For each $t>0$, we have
\begin{align*}
\hat{\varphi}(x)&=\lim_{\sigma\to+\infty}\sup_{s\geq \sigma}T^+_{s+t}\varphi(x)=\lim_{\sigma\to+\infty}\sup_{s\geq \sigma}\sup_{z\in M}h^{z,T^+_s\varphi(z)}(x,t)\\
&=\lim_{\sigma\to+\infty}\sup_{z\in M}h^{z,\sup_{s\geq \sigma}T^+_s\varphi(z)}(x,t),
\end{align*}
where  the last one follows from the monotonicity of $h^{x_0,u_0}(x,t)$ with respect to $u_0$.
Since
\begin{align*}
|\sup_{z\in M}h^{z,\sup_{s\geq \sigma}T^+_s\varphi(z)}(x,t)-T^+_t\hat{\varphi}(x)|&=|\sup_{z\in M}h^{z,\sup_{s\geq \sigma}T^+_s\varphi(z)}(x,t)-\sup_{z\in M}h^{z,\hat{\varphi}(z)}(x,t)|\\
&\leq\sup_{z\in M}|h^{z,\sup_{s\geq \sigma}T^+_s\varphi(z)}(x,t)-h^{z,\hat{\varphi}(z)}(x,t)|\\
&\leq \hat{l}_t\|\sup_{s\geq \sigma}T^+_s\varphi-\hat{\varphi}\|_\infty
\end{align*}
for $\sigma>0$ large enough, by (\ref{5-230}) we have
\[
\hat{\varphi}(x)=T^+_t\hat{\varphi}(x),\quad \forall x\in M,\ \forall t>0,
\]
where $\hat{l}_t$ is the Lipschitz constant of the function $(x_0,u_0,x)\mapsto h^{x_0,u_0}(x,t)$ on $M\times[-K-\|\hat{\varphi}\|_\infty,K+\|\hat{\varphi}\|_\infty]\times M$. Thus, $T^+_t\hat{\varphi}=\hat{\varphi}$ for any $t>0$.

Then we show $\check{\varphi}$ is also a forward weak KAM solution of (\ref{hj}), i.e.,
$\check{\varphi}$ satisfies $T_t^+\check{\varphi}=\check{\varphi}$ for any $t> 0$.
From the uniform boundedness of $\{T^+_t\varphi\}_{t\geq 0}$, it is clear that $\check{\varphi}(x)$ is well-defined. By definition, we have
\[
\lim_{t\to+\infty}\inf_{s\geq t}T^+_s\varphi(x)=\check{\varphi}(x),\quad \forall x\in M.
\]
Since
\[
|\inf_{s\geq t}T^+_s\varphi(x)-\inf_{s\geq t}T^+_s\varphi(y)|\leq\sup_{s\geq t}|T^+_s\varphi(x)-T^+_s\varphi(y)|\leq \kappa_\delta d(x,y),
\quad \forall t>\delta>0,
\]
then
\begin{align}\label{5-23}
\lim_{t\to+\infty}\inf_{s\geq t}T^+_s\varphi(x)=\check{\varphi}(x)
\end{align}
uniformly on $x\in M$.

 It is clear that $T_0^+\check{\varphi}=\check{\varphi}$. For each $t>0$, we have
\begin{align*}
\check{\varphi}(x)=\lim_{\sigma\to+\infty}\inf_{s\geq \sigma}T^+_{s+t}\varphi(x)=\lim_{\sigma\to+\infty}\inf_{s\geq \sigma}\sup_{z\in M}h^{z,T^+_s\varphi(z)}(x,t).
\end{align*}
First of all, we assert that
\begin{equation}\label{excha}
\inf_{s\geq \sigma}\sup_{z\in M}h^{z,T^+_s\varphi(z)}(x,t)=\sup_{z\in M}h^{z,\inf_{s\geq \sigma}T^+_s\varphi(z)}(x,t),
\end{equation}
which implies
\[\check{\varphi}(x)=\lim_{\sigma\to+\infty}\sup_{z\in M}h^{z,\inf_{s\geq \sigma}T^+_s\varphi(z)}(x,t).\]
Note that there exists $K>0$ such that $|T^+_s\varphi(x)|\leq K$ for each $s\geq 0$ and $x\in M$, it follows that
\begin{align*}
|\sup_{z\in M}h^{z,\inf_{s\geq \sigma}T^+_s\varphi(z)}(x,t)-T^+_t\check{\varphi}(x)|&=|\sup_{z\in M}h^{z,\inf_{s\geq \sigma}T^+_s\varphi(z)}(x,t)-\sup_{z\in M}h^{z,\check{\varphi}(z)}(x,t)|\\
&\leq\sup_{z\in M}|h^{z,\inf_{s\geq \sigma}T^+_s\varphi(z)}(x,t)-h^{z,\check{\varphi}(z)}(x,t)|\\
&\leq l_t\|\inf_{s\geq \sigma}T^+_s\varphi-\check{\varphi}\|_\infty,
\end{align*}
where $l_t$ is the Lipschitz constant of the function $(z,u,x)\mapsto h^{z,u}(x,t)$ on $M\times[-K-\|\check{\varphi}\|_\infty,K+\|\check{\varphi}\|_\infty]\times M$. Letting $\sigma\rightarrow+\infty$, for each $t>0$, we have
\[
\check{\varphi}(x)=T^+_t\check{\varphi}(x),\quad \forall x\in M.
\]

It remains to verify the assertion, namely (\ref{excha}) holds. From the monotonicity of $h^{x_0,u_0}(x,t)$ with respect to $u_0$, it follows that given $x\in M$, $t,\sigma>0$, for each $z\in M$ and $s\geq \sigma$,
\[h^{z,T^+_s\varphi(z)}(x,t)\geq h^{z,\inf_{s\geq \sigma}T^+_s\varphi(z)}(x,t),\]
which gives rise to
\[\inf_{s\geq \sigma}\sup_{z\in M}h^{z,T^+_s\varphi(z)}(x,t)\geq\sup_{z\in M}h^{z,\inf_{s\geq \sigma}T^+_s\varphi(z)}(x,t).\]
Next, we show that the inverse inequality holds. By definition, for any $\eps>0$ and $z\in M$, there exists $s_0:=s_0(\eps, z)\geq \sigma$ such that
\[T_{s_0}^+\varphi(z)<\inf_{s\geq \sigma}T_s^+\varphi(z)+\eps.\]
Hence, for each $z\in M$, we have
\[h^{z,T_{s_0}^+\varphi(z)}(x,t)\leq h^{z,\inf_{s\geq \sigma}T_s^+\varphi(z)+\eps}(x,t).\]
It yields
\[\sup_{z\in M}h^{z,T_{s_0}^+\varphi(z)}(x,t)\leq \sup_{z\in M}h^{z,\inf_{s\geq \sigma}T_s^+\varphi(z)+\eps}(x,t).\]
Moreover, we have
\[\inf_{s\geq \sigma}\sup_{z\in M}h^{z,T^+_s\varphi(z)}(x,t)\leq \sup_{z\in M}h^{z,T_{s_0}^+\varphi(z)}(x,t)\leq \sup_{z\in M}h^{z,\inf_{s\geq \sigma}T_s^+\varphi(z)+\eps}(x,t).\]
Given $x\in M$ and $t>0$,  since $M$ is compact, then $h^{z,u}(x,t)$ is uniformly Lipschitz with respect to $u$ if $|u|\leq K$. Moreover, one can find $k>0$  depending only on $t$ and $K$ such that
\[|\sup_{z\in M}h^{z,u_1}(x,t)-\sup_{z\in M}h^{z,u_2}(x,t)|\leq \sup_{z\in M}|h^{z,u_1}(x,t)-h^{z,u_2}(x,t)|\leq k|u_1-u_2|.\]
Let us recall $\{T^+_t\varphi(x)\}_{t\geq 0}$ is bounded by $K$. Then
\[\inf_{s\geq \sigma}\sup_{z\in M}h^{z,T^+_s\varphi(z)}(x,t)\leq\sup_{z\in M}h^{z,\inf_{s\geq \sigma}T_s^+\varphi(z)+\eps}(x,t)\leq \sup_{z\in M}h^{z,\inf_{s\geq \sigma}T_s^+\varphi(z)}(x,t)+k\eps.\]
Letting $\eps\rightarrow 0^+$, we have
\[\inf_{s\geq \sigma}\sup_{z\in M}h^{z,T^+_s\varphi(z)}(x,t)\leq\sup_{z\in M}h^{z,\inf_{s\geq \sigma}T_s^+\varphi(z)}(x,t).\]

So far, we have shown that $\hat{\varphi}$ and $\check{\varphi}$ are forward weak KAM solutions of equation \eqref{hj}.

\medskip
\noindent  {\bf Sufficient condition (3)}:
For $\varphi\in C(M,\mathbb{R})$ satisfying (a') and (b'), both $\hat{\varphi}$ and $\check{\varphi}$ are well defined and belong to $\mathcal{S}_+$. Let $v^*_+$ be a minimal forward weak KAM solution  of equation \eqref{hj}. If  $\varphi\leq v^*_+$, then by the  monotonicity property of $T^+_t$, we get
		\[
		T^+_t\varphi\leq T^+_tv^*_+.
		\]
		Note that $v^*_+\in\mathcal{S}_+$ is equivalent to $T^+_tv^*_+(x)=v^*_+(x)$ for all $(t,x)\in [0,+\infty)\times M$.
		Letting $t\to+\infty$, we have
		 \begin{align}\label{lala}
		 \hat{\varphi},\ \check{\varphi}\leq v^*_+.
		 \end{align}
		 Since $v^*_+$ is a minimal element of $\mathcal{S}_+$, then $ \hat{\varphi}=\check{\varphi}=v^*_+$. It follows that
		 the uniform limit $\lim_{t\to+\infty} T^+_t\varphi$ exists and $\lim_{t\to+\infty} T^+_t\varphi=v_+^*$.

The proof of Main Proposition \ref{77iii} is now complete.

\begin{remark}\label{add}
We aim to show Item (3) in Main Proposition \ref{v} by using Item (2) in Main Proposition \ref{77iii} here.
Let $K=\|\inf_{v_+\in\mathcal{S}_+}v_+\|_\infty+\|u_+\|_\infty+1$. For each $\varphi$ satisfying (a') and (b'),
since \[\hat{\varphi}(x):=\limsup_{t\rightarrow+\infty}T_t^+\varphi(x),\quad \check{\varphi}(x):=\liminf_{t\rightarrow+\infty}T_t^+\varphi(x)\] are forward weak KAM solutions of (\ref{hj}), then there is $T_\varphi>0$ such that
\[
\inf_{v_+\in\mathcal{S}_+}v_+(x)-1\leq \check{\varphi}(x)-1\leq T_t^+\varphi(x)\leq \hat{\varphi}(x)+1\leq u_+(x)+1,
\]
for all $t\geq T_\varphi$ and all $x\in M$.
\end{remark}

\subsection{Proof of Corollary \ref{ad-c}}

We only need to show $\mathcal{VS}(\bar{H})$ of $\bar{H}(x,u,Du)=0$ is a singleton.
By Proposition \ref{i}, it suffices to show the forward weak KAM solution of $\bar{H}(x,-u,-Du)=0$ is unique. Let $H(x,u,p):=\bar{H}(x,-u,-p)$. A direct calculation shows that
\begin{itemize}
\item $(x(t),u(t),p(t))$ is a solution of contact Hamilton equations (\ref{c}) generated by $\bar{H}(x,u,p)$ if and only if $(x(t),-u(t),-p(t))$ is a solution of the contact Hamilton equations (\ref{c}) generated by $-{H}(x,u,p)$;
    \item $(x(t),u(t),p(t))$ is a singular point of the contact vector field generated by $-{H}(x,u,p)$ if and only if it is a singular point of the contact vector field generated by ${H}(x,u,p)$.
    \end{itemize}
It follows that the contact vector field generated by $\bar{H}(x,u,p)$ has no singular points if and only if  the contact vector field generated by ${H}(x,u,p)$ has no singular points.

In view of \cite[Corollary 1.1]{WWY2}, the Aubry set $\tilde{\mathcal{A}}$ of $H(x,u,p)$ is a periodic orbit of the contact Hamiltonian system, denoted by $(x(t),u(t),p(t))$.
Note that for any $v_+\in\mathcal{S}_+$, $\tilde{\mathcal{I}}_{v_+}$ is non-empty.
For any $(x_0,u_0,p_0)\in \tilde{\mathcal{I}}_{v_+}$, by \cite[Proposition 4.1, Lemma 4.7]{WWY2} there is a curve $\gamma:(-\infty,+\infty)\to \mathbb{S}$ such that $\gamma(0)=x_0$, $u_-(x_0)=v_+(x_0)=u_0$, $Du_-(x_0)=Dv_+(x_0)=p_0$,
\begin{align}\label{6600}
u_-(\gamma(t))=v_+(\gamma(t)),\quad \forall t\in\R.
\end{align}
and $(\gamma(t),u_-(\gamma(t)),Du_-(\gamma(t)))=\Phi_t(x_0,u_0,p_0)$.
Let
\[
\bar{u}(t)=u_-(\gamma(t)),\quad \forall t\in\R.
\]
By \cite[Lemma 4.8]{WWY2}, $(\gamma(t),\bar{u}(t))$ is a static curve \cite[Definition 3.2]{WWY2}, which implies that the Aubry set consists of $(\gamma(t),u_-(\gamma(t)),Du_-(\gamma(t)))$. Hence, we deduce that
\[
x(t)=\gamma(t),\quad \forall t\in\R.
\]
In view of \eqref{6600}, we have that $u_-(x(t))=v_+(x(t))$ for all $t\in\R$.
Recall that $(x(t),u(t),p(t))$ is a periodic orbit on $T^*\mathbb{S}\times \R$.
Thus, we get that $u_-=v_+$ everywhere. It means the forward weak KAM solution of $\bar{H}(x,-u,-Du)=0$ is unique.

\appendix\label{AAA}

\section{Weak KAM solutions and solution semigroups}

In this part, we will show that Proposition \ref{pr4.5} holds under the assumption (H1), (H2) and $|\frac{\partial H}{\partial u}|\leq \lambda$.

{\it First of all, we will prove that  a backward weak KAM solution (resp. forward weak KAM solution) is  a fixed point of $T_t^-$ (resp. $T_t^+$).}

\medskip
\noindent  {\bf Step 1}: We will show { $u\in\mathcal{S}_-$ implies $T^-_tu=u$ for all $t\geq 0$}.
By definition, we have
\begin{equation}
u(x)=\inf_{\gm(t)=x}\left\{u(\gm(0))+\int_0^tL(\gm(\tau),u(\gm(\tau)),\dot{\gm}(\tau))d\tau\right\}.
\end{equation}where the infimum is taken among the Lipschitz continuous $C^1$ curves.
We will prove $T_t^-u=u$ for each $t\geq 0$. In fact, it suffices to verify $u(x)\leq T_t^-u(x)$. The converse inequality is similar to be obtained. By contradiction, we assume $u(x)>T_t^-u(x)$. Let $\gm:[0,t]\rightarrow M$ be a minimizer of $T_t^-u$ with $\gm(t)=x$, i.e.
\begin{equation}
T_t^-u(x)=u(\gm(0))+\int_0^tL(\gm(\tau),T_\tau^- u(\gm(\tau)),\dot{\gm}(\tau))d\tau.
\end{equation}
Let $F(\tau)=u(\gm(\tau))-T_\tau^- u(\gm(\tau))$. Since $F(t)>0$ and $F(0)=0$, then one can find $s_0\in [0,t)$ such that $F(s_0)=0$ and $F(s)>0$ for $s\in (s_0,t]$. A direct calculation shows
\[F(s)\leq \lambda\int_{s_0}^sF(\tau)d\tau,\]
which implies $F(s)\leq 0$ for $s\in (s_0,t]$ from Gronwall inequality. It contradicts $F(t)>0$.

\medskip
\noindent  {\bf Step 2}: We will show {$u\in\mathcal{S}_+$ implies $T^+_tu=u$ for all $t\geq 0$}.
By the definition of forward weak KAM solution, we have
\begin{equation}
u(x)=\sup_{\gm(0)=x}\left\{u(\gm(t))-\int_0^tL(\gm(\tau),u(\gm(\tau)),\dot{\gm}(\tau))d\tau\right\}.
\end{equation}where the infimum is taken among the Lipschitz continuous curves.
We will prove $T_t^+u=u$ for each $t\geq 0$. It suffices to check $u(x)\leq T_t^+u(x)$. The converse inequality is similar to be obtained. By contradiction, we assume $u(x)>T_t^+u(x)$. Let $\gm:[0,t]\rightarrow M$  with $\gm(0)=x$ be a curve such that
\begin{equation}
u(x)=u(\gm(t))-\int_0^tL(\gm(\tau), u(\gm(\tau)),\dot{\gm}(\tau))d\tau.
\end{equation}
Let $F(\tau)=u(\gm(\tau))-T_{t-\tau}^+ u(\gm(\tau))$. Since $F(t)=0$ and $F(0)>0$, then one can find $s_0\in (0,t]$ such that $F(s_0)=0$ and $F(s)>0$ for $s\in [0,s_0)$. Note that
\[u(\gm(s))=u(\gm(s_0))-\int_s^{s_0}L(\gm(\tau), u(\gm(\tau)),\dot{\gm}(\tau))d\tau,\]
\[T_{t-s}^+u(\gm(s))\geq T_{t-s_0}^+u(\gm(s_0))-\int_s^{s_0}L(\gm(\tau), T_{t-\tau}^+u(\gm(\tau)),\dot{\gm}(\tau))d\tau.\]
A direct calculation shows
\[F(s)\leq \lambda\int_{s}^{s_0}F(\tau)d\tau,\]
which implies $F(s)\leq 0$ for $s\in [0,s_0]$ from Gronwall inequality. It contradicts $F(0)>0$.

\medskip

\medskip

{\it Next, we will prove that { a fixed point of $T_t^-$ (resp. $T_t^+$) is a backward weak KAM solution (resp. forward weak KAM solution)}.}

\medskip
\noindent  {\bf Step 3}:
We will prove that { a fixed point of $T_t^-$  is a backward weak KAM solution}.
\begin{lemma}\label{la1}
For $u\in C(M,\R)$ and $t\geq 0$, if $u=T_t^-u$ on $M$, then $u\prec L$. Moreover, $u$ is Lipschitz continuous.
\end{lemma}
\begin{proof}
 For each continuous and piecewise $C^1$ curve $\gm:[0,t]\rightarrow M$,  we have
\begin{align*}
u(\gm(t))= T_t^-u(\gm(t))&\leq u(\gm(0))+\int_0^tL(\gm(\tau), T_\tau^- u(\gm(\tau)),\dot{\gm}(\tau))d\tau,\\
&= u(\gm(0))+\int_0^tL(\gm(\tau), u(\gm(\tau)),\dot{\gm}(\tau))d\tau.
\end{align*}
 Then
\begin{equation*}
u(\gm(t))-u(\gm(0))\leq \int_0^tL(\gm(\tau), u(\gm(\tau)),\dot{\gm}(\tau))d\tau.
\end{equation*}
The Lipschitz continuity of $u$ follows from \cite[Lemma 4.1]{WWY2}.
\end{proof}

\begin{lemma}\label{la2}
Let $u\prec L$ and $\gamma:[a,b]\rightarrow M$ be a $(u,L,0)$-calibrated curve.
	 then  $u$ is differentiable at $\gamma(t)$ for each $t\in (a,b)$, $\big(\gamma(t),u(\gamma(t)),p(t)\big)$ satisfies contact Hamilton equations  on $(a,b)$, where \[p(t)=\frac{\partial L}{\partial \dot{x}}(\gamma(t),u(\gamma(t)),\dot{\gamma}(t)).\]
	Moreover, we have
	\[
	\big(\gamma(t+s),u(\gamma(t+s)),Du(\gamma(t+s)\big)=\Phi_{s}\big(\gamma(t),u(\gamma(t)),Du(\gamma(t)\big), \quad\forall t, \ s, \ t+s\in (a,b),
	\]
	and
	\[
	H\big(\gamma(t),u(\gamma(t)),p(t)\big)=0,\quad\forall t\in (a,b).
	\]
\end{lemma}

\begin{proof}
It is a direct result following from Lemma 4.2, Lemma 4.3 and Proposition 4.1 in \cite{WWY2}.
\end{proof}

By Lemma \ref{la1}, $u\prec L$. By $T_t^-u=u$ for each $t\geq 0$, one can find $\bar{\gm}:[0,t]\rightarrow M$ with $\bar{\gm}(0)=x$ such that
\[u(\bar{\gm}(t))-u(\bar{\gm}(0))=\int_0^tL(\bar{\gm}(s),u(\bar{\gm}(s)),\dot{\bar{\gm}}(s))ds.\]
Let $\gm_t(s):=\bar{\gm}(s+t)$, we have
\[u({\gm}_t(0))-u({\gm}_t(-t))=\int_{-t}^0L({\gm}_t(s),u({\gm}_t(s)),\dot{{\gm}}_t(s))ds.\]
Moreover, for each $t'\in (0,t]$, there holds
\begin{equation}\label{e1}
u({\gm}_t(0))-u({\gm}_t(-t'))=\int_{-t'}^0L({\gm}_t(s),u({\gm}_t(s)),\dot{{\gm}}_t(s))ds,
\end{equation}
which means that $\gm_t:[-t,0]\rightarrow M$ is $(u, L, 0)$-calibrated. By Lemma \ref{la1},
$u$ is differentiable at $\gamma_t(s)$ for each $s\in (-t,0)$, $\big(\gamma_t(s),u(\gamma_t(s)),p_t(s)\big)$ satisfies contact Hamilton equations  on $(-t,0)$, where \[p_t(s)=\frac{\partial L}{\partial \dot{x}}(\gamma_t(s),u(\gamma_t(s)),\dot{\gamma}_t(s)).\]
	Moreover, for $\tau,  s,  \tau+s\in (-t,0)$, we have
	\[
	\big(\gamma_t(\tau+s),u(\gamma_t(\tau+s)),Du(\gamma_t(\tau+s)\big)=\Phi_{s}\big(\gamma_t(\tau),u(\gamma_t(\tau)),Du(\gamma_t(\tau)\big),
	\]
	and
	\[
	H\big(\gamma_t(s),u(\gamma_t(s)),p_t(s)\big)=0,\quad\forall s\in (-t,0).
	\]
Note that $\|u\|_\infty\leq C$ and $|\frac{\partial H}{\partial u}|\leq \lambda$, we have
\[H(\gm_t(s),0,p_t(s))\leq \lambda C.\]
By (H2), there exists a constant $K>0$ independent of $t$ such that $\|p_t(s)\|\leq K$ for each $s\in [-t,0]$. Moreover, there exists a constant $K'>0$ independent of $t$ such that $\|\dot{\gm}_t(s)\|\leq K'$ for each $s\in [-t,0]$. It follows that for each $s\in [-t,0]$, $(\gm_t(s),u(\gm_t(s)),\dot{\gm}_t(s))$ is contained in a compact set denoted $\mathcal{K}$ independent of $t$. One can find a sequence $t_n$ such that $(\gm_{t_n}(0),u(\gm_{t_n}(0)),\dot{\gm}_{t_n}(0))$ tends to $(x,u(x),v_0)$ as $n\rightarrow \infty$.

Let $\Gamma_n(s):=(\gm_{t_n}(s),u(\gm_{t_n}(s)),\dot{\gm}_{t_n}(s))$. Then for each $t'\in (0,+\infty)$,
\[\Gamma_n|_{[-t',0]}\rightarrow \Phi_s(x,u(x),v_0)|_{s\in [-t',0]}\]
in the $C^0$ topology.

Let $\gm_\infty(s):=\pi\Phi_s(x,u(x),v_0)$ where $\pi:TM\times\R\rightarrow M$. By Lemma \ref{la2}, we have
\[\Phi_s(x,u(x),v_0)=(\gm_\infty(s),u(\gm_\infty(s)),\dot{\gm}_\infty(s)).\]
By (\ref{e1}), for each $t'\in (0,+\infty)$ and $n$ large enough,
\[u({\gm}_{t_n}(0))-u({\gm}_{t_n}(-t'))=\int_{-t'}^0L({\gm}_{t_n}(s),u({\gm}_{t_n}(s)),\dot{{\gm}}_{t_n}(s))ds,\]
Pass to the limit as $n\rightarrow\infty$, we have
  \begin{equation}\label{eee999}
u(x)-u(\gm_\infty(-t'))=\int_{-t'}^{0}L(\gm_\infty(s),u(\gm_\infty(s)),\dot{\gm}_\infty(s))ds,
\end{equation}
where $\gm_\infty(0)=x$.

\medskip
\noindent  {\bf Step 4}: We will prove that { a fixed point of $T_t^+$  is a forward weak KAM solution}.

Similar to Lemma \ref{la1},
for each $t\geq 0$, if $u= T_t^+u$, then  $u\prec L$.
 By $T_t^+u=u$ for each $t\geq 0$, one can find ${\gm}_t:[0,t]\rightarrow M$ with ${\gm}_t(0)=x$ such that
\[u({\gm}_t(t))-u({\gm}_t(0))=\int_0^tL({\gm}_t(s),u({\gm}_t(s)),\dot{{\gm}}_t(s))ds.\]
Moreover, for each $t'\in (0,t]$, there holds
\begin{equation}\label{e1000}
u({\gm}_t(t'))-u({\gm}_t(0))=\int_{0}^{t'}L({\gm}_t(s),u({\gm}_t(s)),\dot{{\gm}}_t(s))ds,
\end{equation}
which means $\gm_t:[0,t]\rightarrow M$ is $(u, L, 0)$-calibrated. By a similar argument as Step 3 (from (\ref{e1}) to (\ref{eee999})), one can find a $C^1$ curve $\gm_\infty:[0,+\infty)\rightarrow M$ with $\gm_\infty(0)=x$ such that for each $t'\in [0,+\infty)$,
\[u(\gm_\infty(t'))-u(x)=\int^{t'}_{0}L(\gm_\infty(s),u(\gm_\infty(s)),\dot{\gm}_\infty(s))ds.\]
%====================================================================================================================

\vskip 1cm

\noindent {\bf Acknowledgements:} 
%The authors sincerely thank the referee for careful reading and
%invaluable comments, especially on the formulation of Proposition \ref{repaction}, which were very helpful in improving this paper.
Jun Yan wishes to thank Hitoshi Ishii for his kind hospitality and useful discussions at Waseda University
in the spring of 2018.
Kaizhi Wang is supported by NSFC Grant No. 11771283, 11931016.
Lin Wang is supported by NSFC Grant No. 11790273, 11631006.
Jun Yan is supported by NSFC Grant No.  11631006, 11790273.

\end{document}